\documentclass[english,11pt,a4paper]{article}

\usepackage{lipsum}
\usepackage{amsfonts}
\usepackage{graphicx}
\usepackage{epstopdf}
\usepackage{algorithmic}
\ifpdf
  \DeclareGraphicsExtensions{.eps,.pdf,.png,.jpg}
\else
  \DeclareGraphicsExtensions{.eps}
\fi

\usepackage{enumitem}
\setlist[enumerate]{leftmargin=.5in}
\setlist[itemize]{leftmargin=.5in}

\usepackage{amsopn}
\usepackage{stmaryrd}
\usepackage{a4wide}
\usepackage{amsmath,amsthm,epsfig,amssymb,amsbsy}
\usepackage{adjustbox}
\usepackage{enumerate}
\usepackage{comment}
\usepackage{algorithm}
\usepackage{algorithmic}
\usepackage{dsfont}
\usepackage{enumitem}
\usepackage{tikz}
\usepackage{pgfplots}\pgfplotsset{compat=1.13}
\usepackage{float}
\usepackage{caption} 
\usepackage{subcaption}
\usepackage{tikz}
\usepackage{pgfplots}

\RequirePackage{xargs}

\usepackage{hyperref}
\hypersetup{
    colorlinks=true,
    linkcolor=blue,
    filecolor=magenta,
    citecolor =magenta,  
    urlcolor=magenta,
    pdftitle={Anisotropic Proximal Gradient}
    }
\usepackage[capitalize]{cleveref}[0.19]
\crefname{section}{section}{sections}
\crefname{subsection}{subsection}{subsections}
\Crefname{section}{Section}{Sections}
\Crefname{subsection}{Subsection}{Subsections}

\Crefname{figure}{Figure}{Figures}

\crefformat{equation}{\textup{#2(#1)#3}}
\crefrangeformat{equation}{\textup{#3(#1)#4--#5(#2)#6}}
\crefmultiformat{equation}{\textup{#2(#1)#3}}{ and \textup{#2(#1)#3}}
{, \textup{#2(#1)#3}}{, and \textup{#2(#1)#3}}
\crefrangemultiformat{equation}{\textup{#3(#1)#4--#5(#2)#6}}%
{ and \textup{#3(#1)#4--#5(#2)#6}}{, \textup{#3(#1)#4--#5(#2)#6}}{, and \textup{#3(#1)#4--#5(#2)#6}}

\Crefformat{equation}{#2Equation~\textup{(#1)}#3}
\Crefrangeformat{equation}{Equations~\textup{#3(#1)#4--#5(#2)#6}}
\Crefmultiformat{equation}{Equations~\textup{#2(#1)#3}}{ and \textup{#2(#1)#3}}
{, \textup{#2(#1)#3}}{, and \textup{#2(#1)#3}}
\Crefrangemultiformat{equation}{Equations~\textup{#3(#1)#4--#5(#2)#6}}%
{ and \textup{#3(#1)#4--#5(#2)#6}}{, \textup{#3(#1)#4--#5(#2)#6}}{, and \textup{#3(#1)#4--#5(#2)#6}}

\newtheorem{theorem}{Theorem}[section]
\newtheorem{corollary}[theorem]{Corollary}
\newtheorem{lemma}[theorem]{Lemma}
\newlist{lemenum}{enumerate}{1} % also creates a counter called 'propenumi'
\setlist[lemenum]{label=(\roman*), ref=\theproposition(\roman*), font=\rm}
\crefalias{lemenumi}{lemma}

\newlist{propenum}{enumerate}{1} % also creates a counter called 'propenumi'
\setlist[propenum]{label=(\roman*), ref=\theproposition(\roman*), font=\rm}
\crefalias{propenumi}{proposition} 

\newtheorem{definition}[theorem]{Definition}

\theoremstyle{remark}
\newtheorem{remark}[theorem]{Remark}
\newtheorem{example}[theorem]{Example}

\newlist{assumenum}{enumerate}{1} % also creates a counter called 'propenumi'
\setlist[assumenum]{leftmargin=2.1cm,label=(A\arabic*),font=\bfseries}
\creflabelformat{assumenumi}{#2#1#3}
\crefname{assumenumi}{assumption}{assumptions}
\Crefname{assumenumi}{Assumption}{Assumptions}

\DeclareMathOperator*{\argmin}{arg\,min}

\DeclareMathOperator{\dom}{dom}

\DeclareMathOperator{\diam}{diam}

\DeclareMathOperator{\con}{con}

\DeclareMathOperator{\lip}{Lip}
\DeclareMathOperator{\ran}{rge}

\DeclareMathOperator{\ot}{OT}

\DeclareMathOperator{\Div}{Div}

\newcommand{\bR}{\mathbb{R}}

\newcommand{\bE}{\mathbb{E}}

\newcommand{\bN}{\mathbb{N}}

\newcommand{\exR}{\overline{\mathbb{R}}}

\newcommand{\cF}{\mathcal{F}}

\newcommand{\cC}{\mathcal{C}}
\newcommand{\cP}{\mathcal{P}}

\newcommand{\cE}{\mathcal{E}}
\newcommand{\cV}{\mathcal{V}}
\newcommand{\cG}{\mathcal{G}}

\newcommand{\cB}{\mathcal{B}}
\newcommand{\cM}{\mathcal{M}}
\newcommand{\cK}{\mathcal{K}}

\newcommand{\mS}[1]{\mathcal{S}^{#1}}
\newcommand{\mSO}[1]{\mathcal{SO}(#1)}
\newcommand{\mSE}[1]{\mathcal{SE}(#1)}

\newcommand{\gnabla}{\ensuremath{\nabla}_{\!\!g}}
\newcommand{\PE}[1]{\mathbb{R}[#1]}
\newcommand{\PEN}[2]{\mathbb{R}_{#2}[#1]}

\newcommand{\PMN}[1]{\ensuremath{\mathcal{M}^{\dagger}_{n}}}
\newcommand{\PPN}[1]{\ensuremath{\mathcal{P}^{\dagger}_{n}}}

\newcommand*{\tsp}{^{\mkern-1.5mu\mathsf{T}}}
\newcommand*{\mintsp}{^{-\mkern-1.5mu\mathsf{T}}}

\newcommand*{\TM}{\mathrm{T}}

\newcommand{\dd}{\,\mathrm{d}}
\newcommand{\ind}{\iota}

\makeatletter
\let\save@mathaccent\mathaccent
\newcommand*\if@single[3]{%
	\setbox0\hbox{${\mathaccent"0362{#1}}^H$}%
	\setbox2\hbox{${\mathaccent"0362{\kern0pt#1}}^H$}%
	\ifdim\ht0=\ht2 #3\else #2\fi
}
\newcommand*\rel@kern[1]{\kern#1\dimexpr\macc@kerna}
\newcommand*\widebar[1]{\@ifnextchar^{{\wide@bar{#1}{0}}}{\wide@bar{#1}{1}}}
\newcommand*\wide@bar[2]{\if@single{#1}{\wide@bar@{#1}{#2}{1}}{\wide@bar@{#1}{#2}{2}}}
\newcommand*\wide@bar@[3]{%
	\begingroup
	\def\mathaccent##1##2{%
		\let\mathaccent\save@mathaccent
		\if#32 \let\macc@nucleus\first@char \fi
		\setbox\z@\hbox{$\macc@style{\macc@nucleus}_{}$}%
		\setbox\tw@\hbox{$\macc@style{\macc@nucleus}{}_{}$}%
		\dimen@\wd\tw@
		\advance\dimen@-\wd\z@
		\divide\dimen@ 3
		\@tempdima\wd\tw@
		\advance\@tempdima-\scriptspace
		\divide\@tempdima 10
		\advance\dimen@-\@tempdima
		\ifdim\dimen@>\z@ \dimen@0pt\fi
		\rel@kern{0.6}\kern-\dimen@
		\if#31
		\overline{\rel@kern{-0.6}\kern\dimen@\macc@nucleus\rel@kern{0.4}\kern\dimen@}%
		\advance\dimen@0.4\dimexpr\macc@kerna
		\let\final@kern#2%
		\ifdim\dimen@<\z@ \let\final@kern1\fi
		\if\final@kern1 \kern-\dimen@\fi
		\else
		\overline{\rel@kern{-0.6}\kern\dimen@#1}%
		\fi
	}%
	\macc@depth\@ne
	\let\math@bgroup\@empty \let\math@egroup\macc@set@skewchar
	\mathsurround\z@ \frozen@everymath{\mathgroup\macc@group\relax}%
	\macc@set@skewchar\relax
	\let\mathaccentV\macc@nested@a
	\if#31
	\macc@nested@a\relax111{#1}%
	\else
	\def\gobble@till@marker##1\endmarker{}%
	\futurelet\first@char\gobble@till@marker#1\endmarker
	\ifcat\noexpand\first@char A\else
	\def\first@char{}%
	\fi
	\macc@nested@a\relax111{\first@char}%
	\fi
	\endgroup
}
\makeatother

\newenvironment{keywords}{\begin{paragraph}{Key words:}
}
{
\end{paragraph}
}

\newenvironment{MSCcodes}{\begin{paragraph}{MSC codes:}
}
{
\end{paragraph}
}

\numberwithin{equation}{section}

\usepackage[
    backend=bibtex,
    style=ieee,
    sortlocale=de_DE,
    natbib=true,
    url=false, 
    doi=true,
    eprint=false
]{biblatex}
\title{Convex relaxations for large-scale graphically structured nonconvex problems with spherical constraints: An optimal transport approach\thanks{This work was supported by: Research Foundation Flanders (FWO) research projects G081222N, G033822N, G0A0920N; Research Council KU Leuven C1 project No. C14/18/068; EU’s Horizon 2020 research and innovation programme under the Marie Skłodowska-Curie grant agreement No. 953348; Fonds de la Recherche Sci- entifique - FNRS and the Fonds Wetenschappelijk Onderzoek - Vlaanderen under EOS project no 30468160 (SeLMA). E. Laude is supported by the postdoctoral mandate PDMt1/22/023.}}
\author{%
	Robin Kenis\thanks{%
		KU Leuven,
		Department of Electrical Engineering (ESAT-STADIUS),
		Kasteelpark Arenberg 10, 3001 Leuven, Belgium~
		{\tt%
			\href{mailto:robin.kenis@esat.kuleuven.be}{\{robin.kenis,}%
			\href{mailto:emanuel.laude@esat.kuleuven.be}{emanuel.laude,%
			\href{mailto:panos.patrinos@esat.kuleuven.be}{panos.patrinos\}}%
			\href{mailto:robin.kenis@esat.kuleuven.be,emanuel.laude@esat.kuleuven.be,panos.patrinos@esat.kuleuven.be}{@esat.kuleuven.be}%
		}%
	}%
	}
	\and Emanuel Laude\footnotemark[2]
	\and Panagiotis Patrinos\footnotemark[2]%
}%

\begin{document}
%\newbibliography{main}
%\bibliographystyle{main}{unsrt}
\maketitle

\begin{abstract}
In this paper we derive a moment relaxation for large-scale nonsmooth optimization problems with graphical structure and spherical constraints.
In contrast to classical moment relaxations for global polynomial optimization that suffer from the curse of dimensionality we exploit the partially separable structure of the optimization problem to reduce the dimensionality of the search space. Leveraging optimal transport and Kantorovich--Rubinstein duality we decouple the problem and derive a tractable dual subspace approximation of the infinite-dimensional problem using spherical harmonics. This allows us to tackle possibly nonpolynomial optimization problems with spherical constraints and geodesic coupling terms.
We show that the duality gap vanishes in the limit by proving that a Lipschitz continuous dual multiplier on a unit sphere can be approximated as closely as desired in terms of a Lipschitz continuous polynomial.
The formulation is applied to sphere-valued imaging problems with total variation regularization and graph-based \emph{simultaneous localization and mapping} (SLAM). In imaging tasks our approach achieves small duality gaps for a moderate degree.
In graph-based SLAM our approach often finds solutions which after refinement with a local method are near the ground truth solution.
\end{abstract}

\tableofcontents

\section{Introduction}
\subsection{Motivation}
In this paper we develop a convex relaxation approach for optimization problems with a graphical coupling structure whose variables are constrained to live in a manifold. Given a directed graph $\cG=(\cV, \cE)$ the problem takes the form
\begin{equation*} \label{eq:mrf}
\text{minimize}~ \left\{ F(x) \equiv \sum_{u \in \cV} f_u(x_u) + \sum_{(u,v) \in \cE} f_{(u,v)}(x_u, x_v) : x \in \Omega^\cV \right\}\tag{P},
\end{equation*}
where $(\Omega, d)$ is a compact space with metric $d$, $f_u : \Omega \to \bR$ is lower semi-continuous and the coupling terms $f_{(u,v)} : \Omega \times \Omega \to \bR$ are continuous.
In the context of graphical models and exponential families this amounts to \emph{maximum a posteriori} (MAP) inference in a continuous pairwise \emph{Markov random field} (MRF) \cite{WainwrightJordanGraphicalModels}.
Optimization problems of this form appear in a variety of tasks such as language processing \cite{RabinerJuangFundamentalsSpeechRecognition}, image processing \cite{BlakePushmeetMRFVision}, bioinformatics \cite{DurbinEddyAndersBiologicalSequenceAnalysis} or graph-based \emph{simultaneous localization and mapping} (SLAM) \cite{GrisettiKummerleStachnissBugardGraphSLAM} to name a few.
In image processing \cref{eq:mrf} can be used to formulate denoising or inpainting tasks where $(\cV, \cE)$ represents the pixel grid, $\sum_{u \in \cV} f_u(x_u) $ is a data fidelity term and $\sum_{(u,v) \in \cE} f_{(u,v)}(x_u, x_v)$ is a prior such as the \emph{total variation} (TV) that favors images $x$ that are spatially smooth \cite{RudinOsherFatemiDenoising}.

However, in many situations $\Omega$ is a manifold such as the space of rotation matrices or the unit shere and the coupling terms $f_{(u,v)}(x_u, x_v)=d(x_u,x_v)$ amount to the nonsmooth nonconvex geodesic distance. As a result the optimization problem \cref{eq:mrf} is nonconvex and nonsmooth and therefore local optimization techniques such as subgradient descent can get stuck in a poor local optimum. 
Global optimization techniques, instead, suffer from the curse of dimensionality, since the optimization variable lies in a high-dimensional product space $\Omega^\cV$ (in particular $\cV$ is large, e.g., the number of pixels in the image).

As a tradeoff the problem can be reformulated in terms of a more tractable linear program, the \emph{local marginal polytope relaxation} \cite{WainwrightJordanGraphicalModels}. Despite the fact that the local marginal polytope relaxation in general only yields a lower bound to the so-called global marginal polytope relaxation, which is tight but intractable, its solutions are often near globally optimal in practice. 
When $\Omega$ is finite and thus \cref{eq:mrf} is discrete there exists an abundance of efficient methods that solve the local marginal polytope relaxation such as message passing and dual coordinate ascent \cite{KolmogorovConvergentMessagePassing, KomodakisParagiosTziritasMRFOptimization}.
However, for continuous $\Omega$, as considered in this work, the local marginal polytope relaxation amounts to an infinite-dimensional linear program and is thus still intractable.
As a remedy we consider a hierarchy of implementable dual programs obtained by polynomial subspace discretization and show that the duality gap between the dual approximation and the infinite-dimensional local marginal polytope relaxation vanishes if the degrees goes to infinity.
Our approach is related to optimization of (sparse) polynomials \cite{lasserre2006convergent}. However, in contrast to \cite{lasserre2006convergent} our cost functions $f_u$ and $f_{(u,v)}$ are in general not polynomial. In particular we consider geodesic pairwise couplings $f_{(u,v)}=d$ that are not even smooth.

\subsection{Related work}
\paragraph{Polynomial optimization} Semi-definite programming relaxations for polynomial programming have attracted a lot of attention \cite{LasserreGlobalOptimizationMoments, HenrionLasserreBook, parrilo2003semidefinite, BlekhermanParrilloBook}. The idea is to lift nonconvex polynomial optimization problems to equivalent problems into the measure space where the problem is infinite-dimensional, but crucially a linear optimization problem. The infinite-dimensional problem is then approximated by a finite dimensional one by bounding the degree of polynomials to be maximized over, and representing the measures as truncated moment sequences. These relaxations promise convex methods to solve nonconvex polynomial optimization problems, which are in general NP-hard. The backbone of these methods are easily computable nonnegativity certificates such as Putinar's \cite{PutinarPositivstellensatz} or Krivine--Stengele Positivstellensätze \cite{KrivinePositivstellensatz, StenglePositivstellensatz}, or sum-of-squares certificates \cite{LasserreKimChuanBSOS}. The most notable drawback of these very powerful, general approaches is that they suffer greatly from the curse of dimensionality. Approaches to reduce this problem have
been studied from the viewpoint of exploiting structured sparsity \cite{WakiKimSOSStructuredSparsity, WeisserLasserreSparseBSOS}.
\paragraph{Markov Random Fields} The class of weakly coupled problems considered in this work can be interpreted in terms of MAP inference for a continuous, higher order MRF \cite{WainwrightJordanGraphicalModels}, an approach which provides a powerful model of many imaging problems. This class of problems is often studied via a lifting procedure known as the local marginal polytope relaxation, which reformulates the problem as a \emph{linear program} (LP). For MRF problems with continuous state spaces, this LP-relaxation is infinite dimensional, which poses major challenges. The local marginal polytope relaxation is closely related to lifting for spatially continuous variational problems \cite{PockGlobalSolutions}. In \cite{MollenhoffSublabelAccurate} a piecewise linear dual subspace approximation for nonconvex variational problems is considered which can be viewed as a \emph{sublabel-accurate} lifting in the primal space. In \cite{FixAgarwalContinuousGraphicalModel} a piecewise polynomial dual space approximation scheme is proposed as a relaxation of the infinite dimensional problem. The framework studied in \cite{BauermeisterLaudeMollenhoffLiftingLagrangian} builds upon the analysis in \cite{FixAgarwalContinuousGraphicalModel}, even providing a tractable hierarchy for solving MAP-MRF problems using dual space approximation with piecewise polynomial multipliers, albeit restricted to the unit interval. In particular, a refinement for metric coupling terms is proposed exploiting Kantorowich--Rubinstein duality. The goal of the current paper is to continue in this direction by extending the work in \cite{BauermeisterLaudeMollenhoffLiftingLagrangian} to more general metric spaces, with particular focus on the unit sphere.

\paragraph{Manifolds} In order to minimize geodesic coupling terms, we will rephrase the Wasserstein distance with geodesic cost as a maximization problem and approximate its value with a semi-definite program. In \cite{CondatGeodesicDistance} the problem of computing the Wasserstein distance is also rewritten as a semi-definite program, particularly for the case of the circle $\mS{1}$, for which the optimal transport plan is known up to an angle $\alpha$. In \cite{LellmannLiftingMethods} a finite element approach is proposed to solve continuous manifold-valued MRF-problems with more general coupling terms.

\subsection{Summary and contributions}
In \cref{sec:duality}, we revisit the infinite-dimensional local marginal polytope relaxation of \cref{eq:mrf} which is based on marginalization and optimal transport.
As such it is an optimization problem formulated in the space of measures. Leveraging optimal transport duality we consider two equivalent dual formulations which eventually lead to different implementations. The first one is more general but can only be implemented for polynomial coupling terms. The second one leverages Kantorovich--Rubinstein duality to derive a tailored dual formulation for metric and in particular geodesic couplings. This allows us to provide an efficient implementation using a characterisation of Lipschitz continuity of polynomials on manifolds in terms of nonnegativity of polynomials.

In \cref{sec:polynomialduals} we consider hierarchies of dual programs obtained by polynomial subspace discretization and study its convergence properties. Under continuity of the coupling part we show that the duality gap between the general dual discretization and the infinite-dimensional primal relaxation vanishes if the degree $n$ of the polynomial goes to infinity. If, in addition, the underlying graph is a tree we show that the duality gap between the general dual discretization and the original nonconvex problem \cref{eq:mrf} vanishes in the limit. More specifically, for $\Omega$ being a unit sphere we provide explicit convergence rates for the the duality gaps between the dual discretization and the infinite-dimensional relaxation: for the general dual discretization we show that the duality gap vanishes with rate $\mathcal{O}(1/n)$. For the dual formulation tailored to geodesic coupling terms we similarly obtain the rate $\mathcal{O}(1/n)$.

In \cref{sec:semidefinite} we cast the semi-infinite polynomial dual discretization in terms of a finite semi-definite program using nonnegativity of polynomials and \emph{sum-of-squares} (SOS). In particular we describe how the Lipschitz constraint on the multiplier obtained via optimal transport duality can be implemented in terms of linear matrix inequalities.

In \cref{sec:experiments} we present extensive numerical tests considering manifold-valued image processing tasks and graph-based SLAM. In image processing we show that our approach achieves small duality gaps for a moderate degree of the polynomial dual variables. In graph-based SLAM we show that after rounding our approach yields solutions which are close to the ground truth solution while it reduces the memory footprint when compared to classical moment relaxations for sparse polynomial optimization.

%\subsection{Contributions}
%In \cite{BauermeisterLaudeMollenhoffLiftingLagrangian} it is proven that the local marginal polytope relaxation with coupling terms penalizing the Euclidean distance on the unit interval is tight for any graph, by leveraging the closed-form solution to the optimal transport problem for submodular cost studied in \cite{BachSubmodular}. In \cref{sec:duality} we show by means of a simple counterexample that tightness of the local marginal polytope relaxation fails on the unit sphere in any dimension for coupling terms penalizing the geodesic distance, whenever the graph is not a tree (or forest).\\
%The polynomial discretization considered in \cref{sec:polynomialduals} forms an inner approximation to the local marginal polytope relaxation considered in \cref{sec:duality}. We show in \cref{sec:polynomialduals} that the approximation becomes tight as polynomials of arbitrarily large degree are considered. We furhermore study the lowest level of our relaxation and show that it leads to a very natural convexification of the unlifted problem.\\
%In \cref{sec:semidefinite} we reformulate the Lipschitz continuity condition for polynomials restricted to a manifold as a polynomial inequality by leveraging standard results from differential geometry.\\
%We finalize the study in \Cref{sec:experiments} with an experimental evaluation of the considered hierarchy.

\subsection{Notation}
For a compact nonempty set $X \subset \bR^m$, denote by $\cP(X)$ the space of Borel probability measures on $X$ and by $\cM(X)$ the space of Radon measures on $X$. The convex cone of nonnegative Radon measures is denoted by $\cM_+(X)$. Given measurable spaces $(X_1, \Sigma_1)$ and $(X_2, \Sigma_2)$, a nonnegative measure $\mu : \Sigma_1 \to [0,+\infty]$ and a measurable mapping $T:X_1 \to X_2$, $T \sharp \mu:\Sigma_2 \to [0,+\infty]$ is the pushforward of $\mu$ under $T$ defined by: $(T \sharp \mu)(A) = \mu(T^{-1}(A))$ for all $A\in \Sigma_2$.
Let $\delta_x$ denote the Dirac measure centered at $x \in X$. 
Furthermore let $\cC(X)$ be the space of continuous functions on $X$. We write lsc for lower semicontinuous. We will write $\langle \mu, f \rangle=\int_{X} f(x) \dd \mu(x)$, for $f:X \to \bR$, and $f$ lsc, and $\mu \in \cM(X)$. If $d$ is a metric on $X$, we denote this metric space by $(X, d)$. We define the extended real line $\exR = \bR \cup \{+ \infty\}$. We say $f$ admits the modulus of continuity $\omega$ if $|f(y) - f(x)| \leq \omega(d(y, x))$ holds for the increasing extended real-valued function $\omega: [0, +\infty] \rightarrow [0, +\infty]$, with $\omega$ continuous at $0$, and $\omega(0) = 0$. When $\omega$ equals the linear function $x \mapsto L x$, we obtain $|f(y) - f(x)| \leq L d(y, x)$ and write that $f$ is Lipschitz continuous with $\lip f = L$.
%We denote by $C^*=\{y \in \bR^m : \langle y, x \rangle \geq 0, \forall x \in C \}$ the dual cone of $C$.
% and denote the epigraphical closure of some function $f$ by $\cl f$.
For an extended real-valued function $f:\bR^m \to \exR$ let $f^*(y)=\sup_{x \in \bR^m} \langle x, y\rangle -f(x)$ denote the Fenchel conjugate of $f$ at $y$ and $f^{**} := (f^*)^*$ the Fenchel biconjugate. For a set $C\subset \bR^m$ we denote by $\ind_C:\bR^m \to \exR$ the indicator function with $\ind_C(x) =0$ if $x \in C$ and $\ind_C(x) =\infty$ if $x\notin C$ and by $\sigma_C(x) = \sup_{y\in C} \langle x,y \rangle$ the support function of $C$ at $x \in \bR^m$. These notions are defined analogously for the topologically paired spaces $\cM(X)$ and $\cC(X)$.
%The convex hull $\con C$ of a set $C\subset \bR^m$ is the smallest convex set that contains $C$. Equivalently, $\con C$ is the set of all finite convex combinations of points in $C$. With some abuse of notation we write $\con f$ for a function $f$ to denote the largest convex function below $f$. 
We denote by $\PE{x}$ the ring of polynomials and by $\PEN{x}{n} \subset \PE{x}$ the finite-dimensional subspace of polynomials with maximum degree $n$.

\section{Local marginal polytope relaxation and optimal transport duality} \label{sec:duality}
\subsection{Local vs. global marginal polytope}
It is a well-known fact that any nonconvex problem of the form \cref{eq:mrf} can be cast as a convex (even linear) problem considering a reformulation over the space of probability measures $\cP(\Omega^\cV)$ on the product space $\Omega^\cV$:
\begin{equation} \label{eq:global_marginal_polytope}
\min_{x \in \Omega^\cV} ~F(x) = \min_{\mu \in \cP(\Omega^\cV)} ~\langle \mu, F \rangle.
\end{equation}
In the context of Markov random fields and linear programming this is equivalent to the so-called global marginal polytope relaxation.
However, for large $\cV$ the relaxation suffers from a curse of dimensionality even when $\dim \Omega$ is moderate. 

As a remedy, we consider the so-called local marginal polytope relaxation which lifts the optimization variable $x \in \Omega^\cV$ to a product space of probability measures $\mu \in \cP(\Omega)^\cV$ rather than the space of measures on the product space $\mu \in \cP(\Omega^\cV)$ as in \cref{eq:global_marginal_polytope}
\begin{equation} \label{eq:local_marginal_polytope}
\inf_{\mu \in \cP(\Omega)^\cV} \left\{\mathcal{F}(\mu) := \sum_{u \in \cV} \langle \mu_u, f_u \rangle + \sum_{(u,v) \in \cE} \ot_{f_{(u,v)}}(\mu_u,\mu_v) \right\}.
\tag{R-P}
\end{equation}
Here, $\ot$ is the optimal transportation \cite{Kantorovich1960} with marginals $(\mu_u, \mu_v)$ and cost $f_{(u,v)}$ defined by
\begin{equation}	
\label{eq:ot}
\ot_{f_{(u,v)}}(\mu_u, \mu_v) = \inf_{\mu_{(u,v)} \in \Pi(\mu_u, \mu_v)} \langle \mu_{(u,v)}, f_{(u,v)}\rangle.
\end{equation}
The constraint set $\Pi(\mu_u, \mu_v)$ is the set of all Borel probability measures on $\Omega^2$ with prescribed marginals $\mu_u$ and $\mu_v$: 
\begin{equation} \label{eq:marginalization_constraints}
\Pi(\mu_{u}, \mu_v)=\left\{ \mu_{(u,v)} \in \cP(\Omega^2) : \pi_u\sharp \mu_{(u,v)} = \mu_u, \; \pi_v\sharp \mu_{(u,v)} = \mu_v \right\},
\end{equation}
where $\pi_u : \Omega \times \Omega \to \Omega$ corresponds to the canonical projection onto the $u\textsuperscript{th}$ component.

Since $\Pi(\delta_x, \delta_{x'}) = \{ \delta_{(x, x')} \}$ restricting $\mu_u$ (and therefore $\mu_{(u,v)}$) to be Dirac probability measures, the formulation \cref{eq:local_marginal_polytope} is equivalent to the original problem \cref{eq:mrf}. This shows that when optimizing over the larger set of all probability measures we obtain a lower bound to \cref{eq:global_marginal_polytope}, i.e., we have the following important relations:
\begin{equation} \label{eq:relaxation_ordering}
\cref{eq:local_marginal_polytope} \leq \cref{eq:global_marginal_polytope} = \cref{eq:mrf}.
\end{equation}
As is known for the discrete case \cite{WainwrightJordanGraphicalModels}, the infinite-dimensional local marginal polytope relaxation is tight if the coupling graph $(\cV, \cE)$ is a tree. The following theorem is a specialization of \cite[Equation 3.5]{lasserre2006convergent} noting that trees satisfy the \emph{running intersection property} (RIP) \cite[Equation 1.3]{lasserre2006convergent}:
\begin{theorem} \label{thm:thightness_tree}
Let $(\Omega, d)$ be a metric space and assume that $(\cV, \cE)$ is a tree. Then we have that $\cref{eq:local_marginal_polytope} = \cref{eq:global_marginal_polytope}=\cref{eq:mrf}$.
\end{theorem}
\begin{proof}
Denote a minimizing sequence for \cref{eq:local_marginal_polytope} by $(\mu^{k}_{\cV}, \mu^{k}_{\cE}) \in \cP(\Omega)^{\cV} \times \cP(\Omega \times \Omega)^{\cE}$. Since the (RIP) holds for a tree, by \cite[Lemma 6.4]{lasserre2006convergent} there exists for any $k \in \bN$ a measure $\mu^{k} \in \cP(\Omega^{\cV})$ such that
\begin{equation*}
\pi_{u} \sharp \mu^{k} = \mu^{k}_{\cV, u} \quad \textrm{and} \quad \pi_{(u,v)} \sharp \mu^{k} = \mu^{k}_{\cE, (u,v)}.
\end{equation*}
Here $\pi_{u}: \Omega^{\cV} \to \Omega$ denotes the canonical projection onto the $u\textsuperscript{th}$ component, and $\pi_{(u,v)}: \Omega^{\cV} \to \Omega^{2}$ denotes the canonical projection onto the $u\textsuperscript{th}$ and $v\textsuperscript{th}$ components. Then $\mu^{k}$ attains the same cost for \cref{eq:global_marginal_polytope} that $\mu_{\cV}^{k}$ and $\mu_{\cE}^{k}$ attain for $\cref{eq:local_marginal_polytope}$:
\begin{align*}
\langle \mu^{k}, F \rangle &= \langle \mu^{k}, \sum_{u \in \cV} f_{u} \circ \pi_{u} + \sum_{(u,v) \in \cE} f_{(u,v)} \circ \pi_{(u,v)} \rangle \\
&= \sum_{u \in \cV} \langle \pi_{u} \sharp \mu^{k}, f_{u} \rangle + \sum_{(u,v) \in \cE} \langle \pi_{(u,v)} \sharp \mu^{k}, f_{(u,v)} \rangle \\
&= \sum_{u \in \cV} \langle \mu_{\cV, u}^{k}, f_{u} \rangle + \sum_{(u,v) \in \cE} \langle \mu_{\cE, (u,v)}^{k}, f_{(u,v)} \rangle.
\end{align*}
Together with \cref{eq:relaxation_ordering} this implies that $\lim_{k \rightarrow \infty} ~\langle \mu^{k}, F \rangle = \inf_{\mu \in \cP(\Omega^\cV)} ~\langle \mu, F \rangle$, and $\cref{eq:global_marginal_polytope} = \cref{eq:local_marginal_polytope}$.
\end{proof}
More specifically, if $\Omega=[a,b]$ is an interval and the pairwise terms are submodular \cite{BachSubmodular} the local marginal polytope relaxation is tight even for general coupling graphs $(\cV, \cE)$; see \cite[Proposition 2.2]{BauermeisterLaudeMollenhoffLiftingLagrangian}. There is no such result that covers geodesic coupling terms on general manifolds, a counterexample is given in \cref{ex:counterexample_LMP}.\\
\\
The local marginal polytope relaxation represents the lowest stage in a hierarchy of tighter relaxations known as the Sherali--Adams hierarchy \cite{SheraliAdamsHierarchyRelaxations, WainwrightJordanGraphicalModels} in discrete optimization. 
However, higher stages in the hierarchy come at the expense of reduced tractability. Therefore, in this work we always resort to the local marginal polytope relaxation despite a lack of theoretical tightness guarantees for general graphs. This is motivated by the small duality gaps and the near optimality of the rounded solutions observed in our practical evaluations.
\subsection{Duality}
Leveraging optimal transport duality, in this section, we derive two equivalent dual problems of \cref{eq:local_marginal_polytope}: The first one is more general but can only be implemented for polynomial couplings. The second one is tailored towards metric and in particular geodesic coupling terms and is based on Kantorovich--Rubinstein duality. Leveraging equivalent characterizations of Lipschitz continuity on manifolds this allows us to obtain a more tractable formulation in the metric case.

Thanks to \cite[Theorem 1.42]{santambrogio2015optimal} since $f_{(u,v)}$ is lsc and $\Omega$ compact the optimal transportation $\ot_{f_{(u,v)}}(\mu_u,\mu_v)$ with cost $f_{(u,v)}$ between the marginals $\mu_u$ and $\mu_v$ admits a simple dual formulation which amounts to a supremum over continuous functions $(\varphi, \psi) \in \cC(\Omega)$
\begin{equation} \label{eq:ot_dual}
\ot_{f_{(u,v)}}(\mu_u,\mu_v) = \sup_{(\varphi, \psi) \in \cK_{(u,v)}} \langle \mu_u, \varphi \rangle + \langle \mu_v, \psi \rangle,
\end{equation}
that satisfy the constraint
\begin{equation*}
\cK_{(u,v)} =\left\{ (\varphi, \psi) \in \cC(\Omega) : \varphi(x) +\psi(y) \leq f_{(u,v)}(x,y) \quad \forall x,y \in \Omega \right\}.
\end{equation*}
Replacing $\ot_{f_{(u,v)}}$ in \cref{eq:local_marginal_polytope} with its dual formulation \cref{eq:ot_dual} an interchange of $\inf$ and $\sup$ results in
\begin{equation} \label{eq:local_marginal_polytope_dual}
\sup_{(\varphi, \psi) \in (\cC(\Omega)^\cE)^2} ~-\sum_{u \in \cV} \sigma_{\cP(\Omega)}(A_u(\varphi, \psi)-f_u) - \sum_{e \in \cE} \iota_{\cK_{e}}(\varphi_e, \psi_e), \tag{R-D}
\end{equation}
where $A : (\cC(\Omega)^\cE)^2 \to \cC(\Omega)^\cV$ is a linear mapping defined by
\begin{equation} \label{eq:definition_linear_map_A}
A_u(\varphi, \psi) = - \sum_{v : (u,v) \in \cE} \varphi_{(u,v)} - \sum_{v : (v,u) \in \cE} \psi_{(v, u)}.
\end{equation}
Alternatively, the formulation is obtained by dualizing the marginalization constraints \cref{eq:marginalization_constraints} for each edge $(u,v) \in \cE$ with Lagrange multipliers $\varphi_{(u,v)}$ and $\psi_{(u,v)}$.
Thanks to weak duality we have the relation
\begin{equation*}
\cref{eq:local_marginal_polytope_dual} \leq \cref{eq:local_marginal_polytope}.
\end{equation*}
Invoking the Fenchel--Rockafellar duality theorem in infinite dimensions \cite{rockafellar1967duality} and optimal transport duality we show that strong duality holds and in particular $\cref{eq:local_marginal_polytope} = \cref{eq:local_marginal_polytope_dual}$:
\begin{theorem} \label{thm:duality_general}
Let $(\Omega, d)$ be a compact metric space and assume that $f_{e} : \Omega \times \Omega \to \bR$ and $f_u : \Omega \to \bR$ are lsc for all $e \in \cE$ and $u \in \cV$. Then $\cref{eq:local_marginal_polytope} = \cref{eq:local_marginal_polytope_dual}$ and a minimizer of \cref{eq:local_marginal_polytope} exists.
\end{theorem}
The proof follows by standard duality arguments and is deferred to \cref{secsup:duality_general_proof}.\\
\\
If, in addition, $f_e$ is continuous, following the proof of \cite[Proposition 1.11]{santambrogio2015optimal}, we also have existence of dual solutions:
To this end we introduce the so-called $c$-transform
\begin{definition}[$c$-transfrom and $c$-concavity]
Let $f:\Omega \to \exR$ and $c:\Omega \times \Omega \to \bR$. Then we define the $c$-transform and $\bar c$-conjugate of $f$ as
\begin{equation*}
f^c(y) := \inf_{x \in \Omega}~c(x,y) - f(x), \quad f^{\bar c}(x) := \inf_{y \in \Omega}~c(x,y) - f(y).
\end{equation*}
We say that $f \in c-\mathrm{conc}(\Omega)$ is $c$-concave if there exists some $g:\Omega \to \exR$ such that $f(x)= \inf_{y \in \Omega}~c(x,y) - g(y)$ and we say that $f \in \bar c-\mathrm{conc}(\Omega)$ is $\bar c$-concave if there is some $g:\Omega \to \exR$ such that $f(y)= \inf_{x \in \Omega}~c(x,y) - g(x)$.

If $c(x,y)=c(y,x)$, the $\bar c$-conjugate and $c$-conjugate of $f$ are identical.
\end{definition}

\begin{theorem} \label{thm:duality_existence}
Let $(\Omega, d)$ be a compact metric space and assume that for all $e \in \cE$ the pairwise term $f_{e} : \Omega \times \Omega \to \bR$ is continuous with some modulus $\omega_e$ and for all $u \in \cV$ the unary term $f_u : \Omega \to \bR$ is lsc. Then there exists a maximizer $(\varphi,\psi) \in \cC(\Omega)^\cE \times \cC(\Omega)^\cE$ of \cref{eq:local_marginal_polytope_dual} with the property $\varphi_{e} \in f_e-\mathrm{conc}(\Omega)$ and $\psi_{e}=(\varphi_{e})^{f_{e}}$ and thus $\varphi_e,\psi_e$ inherit the modulus of continuity $\omega_{e}$ from $f_{e}$.
\end{theorem}
The proof is by standard arguments and is therefore deferred to \cref{secsup:duality_existence_proof}.\\
\\
In light of \cref{thm:duality_existence} the dual variables $\varphi_e,\psi_e$ can be replaced by their respective $f_e$-transforms without changing the dual cost. 
For $f_{(u,v)}=d$ being a metric, for example the induced Riemannian metric $d_g$ for $(\Omega,g)$ being a compact Riemannian manifold, thanks to \cite[Proposition 3.1]{santambrogio2015optimal} $\varphi_e \in f_e-\mathrm{conc}(\Omega)$ and $\psi_e = (\varphi_e)^{f_e}$ together imply that $\psi_e =-\varphi_e \in \lip(\Omega, d)$ where
\begin{equation*}
\lip(\Omega, d) := \left\{ \varphi : \Omega \to \bR : |\varphi(x) - \varphi(y)| \leq d(x,y) \quad \forall x,y \in \Omega \right\}.
\end{equation*}
Invoking \cref{thm:duality_general} this leads to the following reduced dual formulation of \cref{eq:local_marginal_polytope}:
\begin{equation} \label{eq:local_marginal_polytope_dual_metric}
\sup_{\varphi \in \cC(\Omega)^\cE} ~-\sum_{u \in \cV} \sigma_{\cP(\Omega)} (\Div_u(\varphi)-f_u) - \sum_{e \in \cE} \iota_{\lip(\Omega, d)}(\varphi_e), \tag{mR-D}
\end{equation}
where $\Div : \cC(\Omega)^\cE \to \cC(\Omega)^\cV$ is a graph divergence operator defined by
\begin{equation}
\Div_u(\varphi) = - \sum_{v : (u,v) \in \cE} \varphi_{(u,v)} + \sum_{v : (v,u) \in \cE} \varphi_{(v, u)}.
\end{equation}
We have the following result which is also known as Kantorovich--Rubinstein duality in the context of optimal transport. 
\begin{corollary} \label{thm:duality_metric}
Let $(\Omega, d)$ be a compact metric space and assume that $f_{(u,v)} =d$ and $f_u : \Omega \to \bR$ is lsc. Then $\cref{eq:local_marginal_polytope} = \cref{eq:local_marginal_polytope_dual_metric}$ and there exists a minimizer of \cref{eq:local_marginal_polytope} and a maximizer of \cref{eq:local_marginal_polytope_dual_metric}.
\end{corollary}
As we will see \cref{eq:local_marginal_polytope_dual_metric} can be used to derive a tailored implementation for geodesic coupling terms leveraging a simple characterization of Lipschitz continuity on manifolds whereas \cref{eq:local_marginal_polytope_dual} is applicable for polynomial couplings only.

\section{Polynomial dual approximation} \label{sec:polynomialduals}
\subsection{Metric vs. general dual problem}
In this section we consider polynomial subspace approximations of \cref{eq:local_marginal_polytope_dual} and \cref{eq:local_marginal_polytope_dual_metric}. Although both problems are equivalent in infinite dimensions they lead to different hierachies once we restrict the maximimization to the finite-dimensional subspace of polynomials. This leads us to examine the convergence properties of both hierachies in parallel restricting $\Omega$ to a unit sphere. 

Recall that we denote by $\PE{x}$ the ring of polynomials, $\deg p$ denotes the maximum degree of $p \in \PE{x}$ and for $n \geq 1$ let $\PEN{x}{n} \subset \PE{x}$ be the finite-dimensional subspace of the space of polynomials with some fixed degree $n$.

\subsection{General dual problem}
Firstly, we consider the general dual problem \cref{eq:local_marginal_polytope_dual}.
We restrict the dual variables $\varphi_{e}, \psi_e$ in \cref{eq:local_marginal_polytope_dual} to be polynomials with degree $n$ on the embedding space $\mathbb{E} \supset \Omega$ . Then the discretized dual problem amounts to
\begin{equation} \label{eq:local_marginal_polytope_dual_discrete}
\sup_{\varphi, \psi \in (\PEN{x}{n})^\cE} ~-\sum_{u \in \cV} \sigma_{\cP(\Omega)}(A_u(\varphi, \psi)-f_u) - \sum_{e \in \cE} \iota_{\cK_{e}}(\varphi_e, \psi_e) \tag{R-D\textsuperscript{n}}.
\end{equation}
Next we will show that the duality gap $\cref{eq:local_marginal_polytope} - \cref{eq:local_marginal_polytope_dual_discrete}$ vanishes as $n \rightarrow \infty$.
Intuitively, the maximization process attempts to find the best polynomial approximation to the optimal dual variable whose existence was proved in \cref{thm:duality_general}. More precisely, the duality gap can be bounded in terms of the difference (in the sup-norm topology) between the optimal dual variable and its best polynomial approximation. Invoking the Stone--Weierstrass theorem we thus have the following qualitative result whose proof follows the proof of \cite[Theorem 2]{FixAgarwalContinuousGraphicalModel}.
\begin{theorem} \label{thm:polynomial_duality_gap_vanishes}
Let $(\Omega, d)$ be a compact metric space and let $f_u : \Omega \to \bR$ be lsc and $f_{(u,v)} : \Omega \times \Omega \to \bR$ be continuous. Then we have that the duality gap vanishes, i.e., $\cref{eq:local_marginal_polytope} - \cref{eq:local_marginal_polytope_dual_discrete} \searrow 0$ for $n \rightarrow \infty$. If, in addition, $(\cV,\cE)$ is a tree we have in particular $\cref{eq:mrf} - \cref{eq:local_marginal_polytope_dual_discrete} \searrow 0$ for $n \rightarrow \infty$.
\end{theorem}
For proving the claim we resort to an unconstrained version of \cref{eq:local_marginal_polytope_dual_discrete} obtained via the change of variable $\varphi_e =\gamma_e + \xi_e$ for $\xi_e \in \bR$ for each edge $e \in \cE$.
Then 
\begin{equation*}
-A_u(\varphi, \psi) = \sum_{v : (u,v) \in \cE} \gamma_{(u,v)} + \sum_{v : (v,u) \in \cE} \psi_{(v, u)}+ \sum_{v : (u,v) \in \cE} \xi_{(u,v)} = -A_u(\gamma, \psi) + \sum_{v : (u,v) \in \cE} \xi_{(u,v)}
\end{equation*}
and \cref{eq:local_marginal_polytope_dual_discrete} becomes
\begin{equation*}
\begin{aligned}
\sup_{\substack{\gamma,\psi \in \cC(\Omega)^\cE\\ \xi \in \bR^\cE}} \quad &  \sum_{u \in \cV} \min_{x \in \Omega} f_{u}(x) -A_u(\gamma, \psi)(x) + \sum_{e \in \cE} \xi_e \\
\text{subject to}\quad& f_{e}(x,y) -\gamma_e(x)  -\psi_{e}(y) \geq \xi_e \quad \forall x,y \in \Omega.
\end{aligned}
\end{equation*}
Interpreting $\xi_e$ in terms of a slack variable the problem can be written equivalently in terms of an unconstrained one
\begin{equation} \label{eq:dual_relax_alt}
\sup_{\gamma,\psi \in\cC(\Omega)^\cE} \sum_{u \in \cV} \min_{x \in \Omega} f_{u}(x) - A_u(\gamma, \psi)(x) + \sum_{e \in \cE} \min_{x, y \in \Omega} f_{e}(x,y) - \gamma_{e}(x) -\psi_{e}(y),
\end{equation}
and in particular $\cref{eq:local_marginal_polytope_dual}=\cref{eq:dual_relax_alt}$. Restricting $\gamma_e,\psi_e\in \bR_n[X]$ to the subspace of polynomials with degree $n$ we consider
\begin{equation} \label{eq:dual_relax_poly}
\sup_{\gamma,\psi \in (\PEN{x}{n})^\cE} \sum_{u \in \cV} \min_{x \in \Omega} f_{u}(x) - A_u(\gamma, \psi)(x) + \sum_{e \in \cE} \min_{x, y \in \Omega} f_{e}(x,y) - \gamma_{e}(x) -\psi_{e}(y).
\end{equation}
The rest of the proof is deferred to \cref{secsup:polynomial_duality_gap_vanishes_proof}.\\
\\
Specializing $\Omega$ to $\mS{m}$ we can provide an explicit rate for the duality gap in terms of the moduli of continuity $\omega_{(u,v)}$ of the individual $f_{(u,v)}$ invoking Jackson's theorems for $\mS{m}$ \cite{NewmanShapiroJacksonsTheorem}:
\begin{theorem} \label{thm:GeneralmetricProblemConvergenceRate}
Let $(\Omega, g)$ be the m-dimensional unit sphere $\mS{m}$ with induced Riemannian metric $g$ and let $f_u : \Omega \to \bR$ be lsc and $f_{e} : \Omega \times \Omega \to \bR$ be continuous with modulus $\omega_{e}$ for all $e \in \cE$. For any $m \in \bN$, there exists a constant $C_{m}$ such that the duality gap vanishes with rate $\cref{eq:local_marginal_polytope} - \cref{eq:local_marginal_polytope_dual_discrete} \leq 4 C_{m} \sum_{e \in \cE} \omega_e(1/n)$. If, in addition, $(\cV,\cE)$ is a tree we have in particular $\cref{eq:mrf} - \cref{eq:local_marginal_polytope_dual_discrete}\leq 4 C_{m} \sum_{e \in \cE} \omega_e(1/n)$.
\end{theorem}
The proof is deferred to \cref{secsup:GeneralmetricProblemConvergenceRate_Proof}.
\begin{remark}[(ultra)spherical harmonics]
Polynomials defined on the Euclidean embedding space constrained to the unit sphere $\mS{m}$ can be expressed in a basis of the well-known \emph{(ultra)spherical polynomials}. In the $1$-dimensional case (the unit circle) these are the Chebyshev polynomials which are tightly linked to trigonometric polynomials via the parametrization $\theta \mapsto (\cos(\theta), \sin(\theta))$ of $\mS{1}$. These polynomials are furthermore known to be the eigenfunctions of the Laplace--Beltrami operator on $\mS{m}$.
\end{remark}
Further assuming that $f_{(u,v)}=d_g$ we have the following result:
\begin{corollary} \label{thm:jackson}
Let $(\Omega, g)$ be the $m$-dimensional unit sphere $\mS{m}$ with induced Riemannian metric $d_g$ and let $f_u : \Omega \to \bR$ be lsc and $f_{(u,v)}=d_g$. Then for any $m \in \bN$ there exists a constant $C_{m}$ such that the duality gap vanishes with rate $\cref{eq:local_marginal_polytope} - \cref{eq:local_marginal_polytope_dual_discrete} \leq 4 C_{m} |\cE|m/n$. If, in addition, $(\cV,\cE)$ is a tree we have in particular $\cref{eq:mrf} - \cref{eq:local_marginal_polytope_dual_discrete} \leq 4 C_{m} |\cE|/n$.
\end{corollary}
This corollary proves convergence of the general relaxation for geodesic penalty terms, which is equivalent to the metric problem \cref{eq:local_marginal_polytope_dual_metric} in light of \cref{thm:duality_metric}. Unfortunately, \cref{eq:local_marginal_polytope_dual_discrete} with geodesic coupling terms is not readily implementable. In contrast, the metric problem is structurally simpler, involving only a Lipschitz constraint on the dual variable $\varphi$. In the next subsection we will therefore study a polynomial optimization hierarchy for \cref{eq:local_marginal_polytope_dual_metric}, for which we prove convergence with the slightly better rate $2 C_{m} |\cE|/n$.
\subsection{Metric dual problem} \label{ssec:metric_dual_problem}
In this section we consider a hierarchy of polynomial discretizations for the metric dual formulation \cref{eq:local_marginal_polytope_dual_metric} which amounts to
\begin{equation} \label{eq:local_marginal_polytope_dual_metric_discrete}
\sup_{\varphi \in (\PEN{x}{n})^\cE} ~-\sum_{u \in \cV} \sigma_{\cP(\Omega)} (\Div_u(\varphi)-f_u) - \sum_{e \in \cE} \iota_{\lip(\Omega, d)}(\varphi_e), \tag{mR-D\textsuperscript{n}}
\end{equation}
All results in this section are specialized to the unit sphere $\Omega = \mS{m}$. Next we state the main result of this section which shows that for $\Omega$ being the $m$-dimensional unit sphere the duality gap vanishes at rate $\mathcal{O}(1/n)$ where $n$ is the degree of the dual variable:
\begin{theorem} \label{thm:geodesic_approximation}
Let $(\Omega, g)$ be the $m$-dimensional unit sphere $\mS{m}$ with induced Riemannian metric $g$ and let $f_u : \Omega \to \bR$ be lsc and $f_{(u,v)}=d_g$. For any $m \in \bN$ there exists a constant $C_{m}$ such that the duality gap vanishes sublinearly at rate $\cref{eq:local_marginal_polytope} - \cref{eq:local_marginal_polytope_dual_metric_discrete} \leq 2 |\cE| C_{m} \tfrac{1}{n}$. If, in addition, $(\cV,\cE)$ is a tree we have in particular $\cref{eq:mrf} - \cref{eq:local_marginal_polytope_dual_metric_discrete} \leq 2 |\cE| C_{m} \tfrac{1}{n}$.
\end{theorem}
%For $\Omega$ being the unit circle, i.e., $m=1$ the rate is worse than the rate obtained for the general dual formulation \cref{eq:local_marginal_polytope_dual} provided in \cref{thm:jackson}. This is caused by the fact that in the metric dual problem the maximimization process is restricted to the space of Lipschitz polynomials: Although in both problems, \cref{eq:local_marginal_polytope_dual} and \cref{eq:local_marginal_polytope_dual_metric}, there is a Lipschitz continuous optimal dual variable $\varphi$, its best polynomial approximation is potentially non-Lipschitz. Restricting the optimization process to Lipschitz polynomials the best approximations are infeasible for the metric dual problem which leads to a potentially worse convergence rate.\\
%\\
We will prove this theorem in a style very similar to the proof of \cref{thm:GeneralmetricProblemConvergenceRate}, again using the polynomial approximants defined in \cite{NewmanShapiroJacksonsTheorem}. To generalize that proof to one for the metric problem formulation we require that for any degree $n$, there exist polynomial approximants $\varphi_{n} \in \PEN{x}{n}$ to $\varphi$ that are Lipschitz continuous if $\varphi$ is, with $\lip \varphi_{n} = \lip \varphi$, for all $n \in \bN$. We will prove this fact first.\\
\\
For any $m \in \bN$, there exists a measure $\mu$ on $\mS{m}$ that is invariant to the left group action of $\mSO{m+1}$, where $\mSO{m+1}$ is the special orthogonal group in dimension $m+1$ which can be identified with the space of rotation matrices in $\bR^{(m+1) \times (m+1)}$:
\begin{equation*}
\mu(B) = \mu(r^{-1}B), \quad \forall B \in \cB(\mS{m}), ~ r \in \mSO{m+1}.
\end{equation*}
In \cite{NewmanShapiroJacksonsTheorem} it is proved that for all $n \in \bN:$ there exists a $\kappa_{n} \in \PEN{x}{n}$ such that $\kappa_{n}(\tau)$ is nonnegative for $\tau \in [-1, +1]$, and for any $x \in \mS{m}$,
\begin{equation} \label{eq:KappaIntegratesToOne}
\int_{\mS{m}} \kappa_{n}(\langle x, y \rangle) \dd \mu(y) = 1.
\end{equation}
Then an approximant $\varphi_{n}$ of $\varphi$ can be computed as
\begin{equation} \label{eq:definition_K_operator}
\varphi_{n}(x) = [K_{n} \varphi] (x)  := \int_{\mS{m}} \varphi(y) \kappa_{n}(\langle x, y \rangle) \dd \mu(y).
\end{equation}
It can be shown \cite[Theorem 3.3]{RagozinApproximationSpheres} that for all $m \in \bN$ there exists a $C_{m} \in \bR$ such that
\begin{equation} \label{eq:linear_rate_infinity_norm}
\| \varphi - K_{n} \varphi \|_{\infty} \leq C_{m} \tfrac{L}{n},
\end{equation}
which is the approximation result also used in \cref{thm:GeneralmetricProblemConvergenceRate}. Before proving that $K_{n} \varphi$ is Lipschitz continuous with $\lip \varphi = L$, we require a lemma:
\begin{lemma} \label{lma:GeodesicRotation}
For any $x, x' \in \mS{m}$ there exists $r \in \mSO{m+1}$ such that $x' = r x$ and $\displaystyle \min_{y \in \mS{m}} \langle y, r y \rangle = \langle x, x' \rangle$.
\end{lemma}
\begin{proof}
We will identify any element of $y \in \mS{m}$ with points in $y \in \bR^{m+1}$ such that $\|y\| = 1$, and $r \in \mSO{m+1}$ with a rotation matrix $R \in \bR^{(m+1) \times (m+1)}$. Assume $x$ and $x'$ are not linearly dependent, and let $U_{S} = [x, \tfrac{1}{\sqrt{1 - \langle x, x' \rangle^{2}}} x' - \tfrac{\langle x, x' \rangle}{\sqrt{1 - \langle x, x' \rangle^{2}}} x ]$, which has orthonormal columns that span $S = \ran([x, x'])$. Furthermore, let $U_{S^{\perp}}$ be any matrix such that its columns are an orthonormal basis for $S^{\perp}$. One can define
\begin{equation*}
R = \begin{bmatrix}
U_{S^{\perp}} & U_{S}
\end{bmatrix} \begin{bmatrix}
I & 0 & 0 \\
0 & \langle x, x' \rangle & - \sqrt{1 - \langle x, x' \rangle} \\
0 & \sqrt{1 - \langle x, x' \rangle} & \langle x, x' \rangle
\end{bmatrix} \begin{bmatrix}
U_{S^{\perp}} & U_{S}
\end{bmatrix} \tsp =: U \tilde{R} U\tsp
\end{equation*}
and readily check that $R x = x'$. Furthermore, $R$ is unitarily similar to the rotation matrix $\tilde{R}$, hence $R$ is itself a rotation matrix. Constructing $\tfrac{R + R\tsp}{2}$, we have
\begin{equation*}
\tfrac{R + R\tsp}{2} = \begin{bmatrix}
U_{S^{\perp}} & U_{S}
\end{bmatrix} \begin{bmatrix}
I & 0 & 0 \\
0 & \langle x, x' \rangle & 0 \\
0 & 0 & \langle x, x' \rangle
\end{bmatrix} \begin{bmatrix}
U_{S^{\perp}} & U_{S}
\end{bmatrix} \tsp
\end{equation*}
which is an eigendecomposition of $\tfrac{R+R\tsp}{2}$. Since for all $x, x' \in \mS{m}$ we have $\langle x, x' \rangle \leq 1$, the smallest eigenvalue of $\tfrac{R+R\tsp}{2}$ is $\lambda_{\text{min}}(\tfrac{R+R\tsp}{2}) = \langle x, x' \rangle$. Finally, we have the Rayleigh quotient
\begin{equation*}
\min_{y \in \mS{m}} \langle y, R y \rangle = \min_{y \in \mS{m}} \langle y, \tfrac{R + R \tsp}{2} y \rangle = \lambda_{\text{min}}(\tfrac{R + R\tsp}{2}) = \langle x, x' \rangle.
\end{equation*}
If $x$ and $x'$ are linearly dependent, the argument holds via a limiting process by replacing $x'$ by $\sqrt{1 - \varepsilon^{2}} x' + \varepsilon r$ where $r \in \mS{m}$ can be chosen random, when $\varepsilon \rightarrow 0$. In particular, when $x = x'$, we get $R = I$ and we have $\langle y, R y \rangle = \| y \|^{2} = 1 = \langle x, x' \rangle$ for all $y \in \mSO{m}$. If $x = -x'$, then $R$ can be any rotation such that $R x = x'$ (depending on the choice of $r$), but since $R y \in \mS{m}$ for any $y \in \mS{m}$, we have by Cauchy-Schwartz $\langle y, R y \rangle \geq - \| y \| \| Ry \| \geq -1 = \langle x, x' \rangle$ with equality for $y = x$.
\end{proof}
We prove that this approximant is Lipschitz continuous on $\mS{m}$:
\begin{theorem} \label{thm:polynomial_approximation_is_Lipschitz}
If $\varphi: \mS{m} \rightarrow \bR$ is Lipschitz continuous with $\lip \varphi = L$, then $\varphi_{n} = K_{n} \varphi$ is Lipschitz continuous with $\lip \varphi_{n} = L$.
\end{theorem}
\begin{proof}
Denoting the action of $r \in \mSO{m+1}$ on a function $\varphi: ~\mS{m} \rightarrow \bR$ by $r \cdot \varphi(x) := \varphi(r^{-1} x)$, \cite[Proposition 3.1]{RagozinApproximationSpheres} shows that $K_{n}$ is an $\mSO{m+1}$-invariant operator: 
\begin{equation*}
r \cdot K_{n} \varphi = K_{n} (r \cdot \varphi), \quad \forall r \in \mSO{m+1}.
\end{equation*}
For any $x, x' \in \mS{m}$, choose $r \in \mSO{m+1}$ as in \cref{lma:GeodesicRotation}, for which we have $r x = x'$ and that
\begin{equation} \label{eq:lemma_minimum_property}
\min_{z \in \mS{m}} z\tsp r z = \langle x, x' \rangle.
\end{equation}
Then 
\begin{equation*}
\varphi_{n}(x') = [K_{n} \varphi] (x') = [K_{n} \varphi] (r^{-1} x) = r \cdot [K_{n} \varphi] (x) = [K_{n} (r \cdot \varphi)](x),
\end{equation*}
and $\varphi_{n}(x) = [K_{n} \varphi](x)$. By linearity of $K_{n}$, we then obtain
\begin{equation*}
\varphi_{n}(x') - \varphi_{n}(x) = [K_{n} (r \cdot \varphi - \varphi)] (x).
\end{equation*}
Writing $K_{n}$ out in full, we obtain
\begin{align}
|\varphi_{n}(x') - \varphi_{n}(x)| &= \Big| \int_{\mS{m}} \varphi(r^{-1} y) - \varphi(y) \kappa_{n}(\langle x, y \rangle) \dd \mu(y) \Big| \notag \\
&\leq \int_{\mS{m}} |\varphi(r^{-1} y) - \varphi(y)| \kappa_{n}(\langle x, y \rangle) \dd \mu(y) \notag  \\
&\leq L \int_{\mS{m}} d_{g}(r^{-1} y, y) \kappa_{n}(\langle x, y \rangle) \dd \mu(y) \notag  \\
&\leq L \| d_{g} (r^{-1} \cdot, \cdot) \|_{\infty} \int_{\mS{m}} \kappa_{n}(\langle x, y \rangle) \dd \mu(y) \notag  \\
&= L \| d_{g} (r^{-1} \cdot, \cdot) \|_{\infty}. \label{eq:BoundInfinityNorm}
\end{align}
The first inequality follows since $\kappa_{n}(\tau) \geq 0$, for any $\tau \in [-1, +1]$, the second inequality follows from the Lipschitz continuity of $\varphi$. The final equality follows from \cref{eq:KappaIntegratesToOne}.\\
\\
The result is proved if $\| d_{g} (r^{-1} \cdot, \cdot) \|_{\infty} = d_{g}(x, x')$. For any $m$, we have that $d_{g}(x, x') = \arccos(\langle x, x' \rangle)$ for $x, x' \in \mS{m}$. Hence $d_{g}(r^{-1} y, y) = \arccos( y\tsp r\mintsp y)$. Since $\arccos$ is a monotonically decreasing function, we have
\begin{equation} \label{eq:ArccosMaximization}
\max_{y \in \mS{m}} \big\{ \arccos(y\tsp r\mintsp y) \big\}  = \arccos \Big( \min_{y \in \mS{m}} \big\{ y\tsp r\mintsp y \big\} \Big) = \arccos\Big( \min_{z \in \mS{m}} \big\{ z\tsp r z \big\} \Big),
\end{equation}
where the last equality follows from the change of variables $y = r z$. We may write $\min$ and $\max$ as we are optimizing a bounded function on a compact set. By \cref{eq:lemma_minimum_property} together with \cref{eq:ArccosMaximization} this yields
\begin{equation*} 
\| d_{g}(r^{-1} \cdot, \cdot) \|_{\infty} = \max_{y \in \mS{m}} \big\{ \arccos(y\tsp r\mintsp y) \big\}  = \arccos \big( \langle x, x' \rangle \big) = d_{g}(x, x').
\end{equation*}
and using \cref{eq:BoundInfinityNorm} we get
\begin{equation*}
|\varphi_{n}(x') - \varphi_{n}(x)| \leq d_{g}(x, x') \qedhere
\end{equation*}
\end{proof}
We finally write down the proof of \cref{thm:geodesic_approximation}
\begin{proof}[Proof of \cref{thm:geodesic_approximation}]
There exists a maximizer of \cref{eq:local_marginal_polytope_dual_metric} by \cref{thm:duality_metric} which we denote by $\psi \in \lip_{d}(\Omega)^{\cE}$. Then for any $\varphi$, we have
\begin{align*}
- \sigma_{\cP(\Omega)}\big[ \Div_{u} (\psi) - f_{u} \big] &=
\min_{x \in \Omega} \big[ f_{u} - \Div_{u} (\psi) \big] (x) \\
&= \min_{x \in \Omega} f_{u}(x) - \Div_{u} \varphi(x) - \Div_{u} \psi(x) + \Div_{u} \varphi(x) \\
&\leq \min_{x \in \Omega} f_{u}(x) - \Div_{u} \varphi(x) + \| \Div_{u} (\varphi - \psi) \|_{\infty} \\
&\leq \sigma_{\cP(\Omega)}\big[ \Div_{u} (\varphi) - f_{u}  \big] + \| \Div_{u} (\varphi - \psi) \|_{\infty}
\end{align*}
Since we are maximizing over a strictly smaller set in \cref{eq:local_marginal_polytope_dual_metric_discrete} as compared to \cref{eq:local_marginal_polytope_dual_metric}, clearly $\cref{eq:local_marginal_polytope_dual_metric} \geq \cref{eq:local_marginal_polytope_dual_metric_discrete}$. On the other hand we can now bound
\begin{align}
\cref{eq:local_marginal_polytope_dual_metric} - \cref{eq:local_marginal_polytope_dual_metric_discrete} &\leq \sum_{u \in \cV} \| \Div_{u} (\varphi - \psi) \|_{\infty} \notag \\
&\leq \sum_{u \in \cV} d_{u} \sup_{e \in \cE} \| \varphi_{e} - \psi_{e} \|_{\infty} \notag \\
&\leq 2 | \cE | \sup_{e \in \cE} \| \varphi_{e} - \psi_{e} \|_{\infty} \label{eq:sup_norm_bound_dmRD_in_mRD}
\end{align}
with $d_{u}$ denoting the degree of the vertex $u$.\\
\\
By \cref{thm:polynomial_approximation_is_Lipschitz} and \cref{eq:linear_rate_infinity_norm} there exists a sequence of Lipschitz continuous polynomials $\varphi_{n}$ indexed by degree $n$ such that $\| \varphi - \varphi_{n} \|_{\infty} \leq C_{m} \tfrac{L}{n}$, with $\lip \varphi_{n} = \lip \varphi$, for all $n \in \bN$. Now to obtain a lower bound of \cref{eq:local_marginal_polytope_dual_metric_discrete}, choose for every edge $e \in \cE$ that $\varphi_{e}$ is such a polynomial $\varphi_{e} = (\psi_{e})_{n}$ in the sequence approximating $\psi_{e}$. We obtain that
\begin{equation*}
\| \varphi_{e} - \psi_{e} \|_{\infty} \leq C_{m} \tfrac{L}{n},
\end{equation*}
which together with \cref{eq:sup_norm_bound_dmRD_in_mRD} yields
\begin{equation*} \label{bound_dmRD_in_mRD}
\cref{eq:local_marginal_polytope_dual_metric} - \cref{eq:local_marginal_polytope_dual_metric_discrete} \leq 2 |\cE| C_{m} \tfrac{L}{n}. \qedhere
\end{equation*}
\end{proof}

\subsection{Degree $1$ case} \label{ssec:degree_1_case}

In practice of course, we cannot let $n$ tend to infinity, and we are interested in the solutions of the convex relaxation for bounded degree $n$. It is instructive to take a closer look at the case when the optimization problem is restricted to affine multipliers, i.e. when the degree of the relaxation is $n = 1$. It can be shown that this first level of the relaxation hierarchy reduces to a natural convexification of \cref{eq:mrf}. The first part of the following theorem is a specialization of \cite[Theorem 1]{FixAgarwalContinuousGraphicalModel}:
\begin{theorem}[degree $1$ relaxation]
Let the degree $n=1$. Then \hyperref[eq:local_marginal_polytope_dual_discrete]{\normalfont (R-D\textsuperscript{1})} is equivalent to
\begin{equation} \label{eq:degree_1_relaxation_general}
\inf_{x \in (\bR^{M})^{\cV}} \Big\{ \sum_{u \in \cV} f^{**}_{u}(x_{u}) + \sum_{(u,v) \in \cE} f^{**}_{(u, v)}(x_{u}, x_{v}) \Big\},
\end{equation}
where $f_{u}$ and $f_{(u,v)}$ are to be understood as extended real-valued functions on the embedding space, i.e. $f_{u}: \bR^{M} \to \exR$, with $f_{u}(x) = + \infty, ~ x \not\in \Omega$ and finite otherwise, and similarly for $f_{(u,v)}: \bR^{M} \times \bR^{M} \to \exR$. In particular, $\dom(f^{**}_{u}) = \con(\Omega)$.
Furthermore, when $\Omega = \mS{m}$, \hyperref[eq:local_marginal_polytope_dual_metric_discrete]{\normalfont (mR-D\textsuperscript{1})} is equivalent to
\begin{equation} \label{eq:degree_1_relaxation_metric}
\inf_{x \in (\bR^{M})^{\cV}} \Big\{ \sum_{u \in \cV} f^{**}_{u}(x_{u}) + \sum_{(u,v) \in \cE} \| x_{u} - x_{v} \| \Big\},
\end{equation}
where $\dom(f^{**}_{u}) = \mathcal{B}^{m}$, the $m$-dimensional unit ball.
\end{theorem}
\begin{figure}[h!]
 \centering
 \begin{subfigure}[b]{0.45\textwidth}
 \centering
 \includegraphics[width=\textwidth]{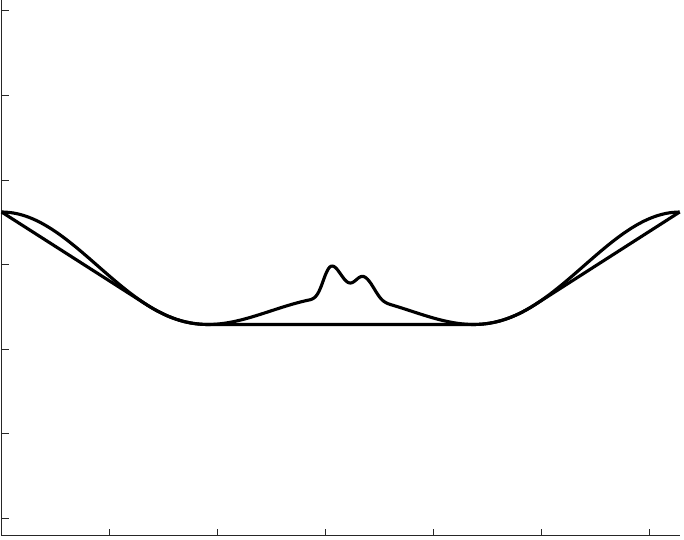}
 \caption{}
 \end{subfigure}
 \begin{subfigure}[b]{0.45\textwidth}
 \centering
 \includegraphics[width=\textwidth]{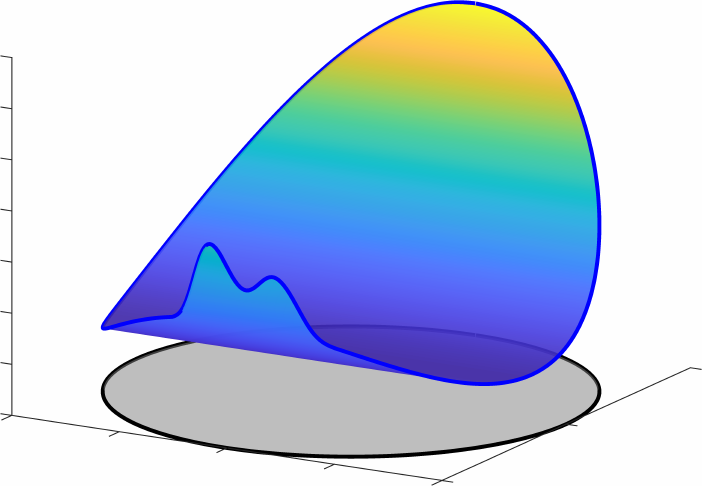}
 \caption{}
 \end{subfigure}
 \caption{Naive convexification via a Euclidean parametrization (a) vs. convex hulls obtained by convexification over the manifold (b). The upper black curve in (a) depicts the graph of the nonconvex periodic cost plotted on $[0, 2\pi)$. The lower black curve corresponds to the graph of its largest convex under-approximation, its convex hull. (b) is obtained by prescribing the same nonconvex periodic cost on the unit-circle (blue curve) and computing its convex hull (colored surface). It can be seen that the proposed convexification (colored surface) preserves cost function when restricted to the manifold (black circle), whereas the naive convexification loses the overall shape of the cost function. This happens at the cost of lifting the one-dimensional domain $[0, 2\pi)$ of the optimization problem to the convex hull of the manifold (gray surface).}
 \label{fig:convexification}
 \end{figure}
\begin{proof}
It can easily be verified that both \cref{eq:local_marginal_polytope_dual_discrete} and \cref{eq:local_marginal_polytope_dual_metric_discrete} are invariant to the constant part of the dual multiplier $\gamma$ and $\psi$ 
(of $\varphi$), so it is natural to further restrict the optimization to linear functionals; i.e. when the constant part(s) of $\gamma$ and $\psi$ (or $\varphi$) equal(s) zero. Denote by $M$ the dimension of the Euclidean embedding space, we restrict the optimization to linear functionals $\gamma, \psi \in \bR^{M}$. Define $F$, $G$ as
\begin{equation*}
F(x) = \sum_{u \in \cV} f_{u}(x_{u}) \quad \text{and} \quad G(x, y) = \sum_{e \in \cE} f_{e}(x_{e}, y_{e})
\end{equation*}
and let $A$ be defined as in \cref{eq:definition_linear_map_A}. We have that
\begin{equation*}
- F^{*} \big(A (\gamma, \psi) \big) = \sum_{u \in \cV} - f^{*}_{u} \big( A_{u} (\gamma, \psi) \big) = \sum_{u \in \cV} \min_{x \in \bR^{M}} f_{u}(x) - \langle A_{u}(\gamma, \psi), x \rangle.
\end{equation*} 
The first equality follows from the properties of the Fenchel conjugate for separable $F$. Similarly for $G^{*}$, we obtain
\begin{equation*}
- G^{*}\big(\gamma,\psi\big) = \sum_{e \in \cE} - f^{*}_{e} (\gamma_{e}, \psi_{e}) = \sum_{e \in \cE} \min_{x, y \in \bR^{M}} f_{e}(x, y) - \langle \gamma_{e}, x \rangle - \langle \psi_{e}, y \rangle
\end{equation*} 
As in the proof for \cref{thm:polynomial_duality_gap_vanishes}, we may rewrite \cref{eq:local_marginal_polytope_dual_discrete} as \cref{eq:dual_relax_poly}, where, in light of the above derivation, restriction to the case $n=1$ yields
\begin{align*}
\text{\hyperref[eq:local_marginal_polytope_dual_discrete]{\normalfont (R-D\textsuperscript{1})}} &= \sup_{\gamma, \psi \in (\bR^{M})^{\cE}} F^{*}\big(A(\gamma,\psi)\big) + G^{*}\big( \gamma,\psi \big) = \inf_{x \in (\bR^{M})^{\cV}} F^{**}(x) + G^{**}(-A^{*}x) \\ &= \inf_{x \in (\bR^{M})^{\cV}} \sum_{u \in \cV} f_{u}^{**}(x) + \sum_{(u, v) \in \cE} f^{**}_{(u, v)}(x_{u}, x_{v})
\end{align*}
where the first equality is an application of Fenchel--Rockafellar duality, and the second is a straightforward rewriting using the definition of the operator $A$. \\
\\
In the case of \cref{eq:local_marginal_polytope_dual_metric_discrete} when $\Omega = \mS{m}$, the dimension of the embedding space is $M = m+1$. We have that any linear functional $\varphi$ which is Lipschitz continuous with $\lip_{(\bR^{M}, d)} \varphi = 1$ on the embedding space is necessarily Lipschitz continuous with $\lip_{(\mS{m}, g)} \varphi = 1$ on $\mS{m}$, since $1 \geq \| \nabla \varphi \| \geq \| \nabla_{g} \varphi \|$. On the other hand, for any linear functional $\varphi$, there exists a $x \in \mS{m}$ such that $\langle \varphi, x \rangle = 0$ where $\nabla_{g} \varphi(x) = \nabla \varphi(x)$. Therefore, for $\varphi: ~\mS{m} \rightarrow \bR$ we have $\iota_{\lip(\mS{m}, g)}(\varphi_e) = \iota_{\lip(\bR^{M}, \| \cdot \|)}(\varphi_e)$ when restricting to $\varphi_{e} \in \bR^{M}$. Defining $F$ as before, we have
\begin{equation*}
-F^{*}(\Div \varphi) = -\sum_{u \in \cV} f_{u}^{*}(\Div_{u} \varphi) = \sum_{u \in \cV} - \sigma_{\cP(\Omega)}(\Div_{u} \varphi - f_u)
\end{equation*}
For the metric case, we have that
\begin{equation*}
-G^{*}(\varphi) = \sum_{e \in \cE} \iota_{\lip(\mS{m}, g)}(\varphi_e) = \sum_{e \in \cE} \iota_{\lip(\bR^{M}, \| \cdot \|)}(\varphi_e)
\end{equation*}
where the last equality holds only in the case $n=1$, when $\varphi_{e} \in \bR^{M}$. Then
\begin{align*}
\text{\hyperref[eq:local_marginal_polytope_dual_metric_discrete]{\normalfont (mR-D\textsuperscript{1})}} &= \sup_{\varphi \in (\bR^{M})^{\cE}} F^{*}\big(\Div \varphi \big) + G^{*}\big( \varphi \big) = \inf_{x \in (\bR^{M})^{\cV}} F^{**}(x) + G^{**}(-\Div^{*}x) \\
&= \inf_{x \in (\bR^{M})^{\cV}} \sum_{u \in \cV} f_{u}^{**}(x) + \sum_{(u, v) \in \cE} \| x_{u} - x_{v} \|.
\end{align*}
where the first equality is an application of Fenchel--Rockafellar duality, and the second is a straightforward rewriting using the definition of the graph divergence operator $\Div$ and that for $y \in \bR^{M}$ one has $\iota_{\lip(\bR^{M}, \| \cdot \|)}^{*}(y) = \iota^{*}_{\cB^{M-1}}(y) = \| y \|$.
\end{proof}
\begin{remark}[relative error degree $1$ case]
Based on this result, it can readily be checked that for two Dirac masses $\delta_{x}$ and $\delta_{y}$, the relative error $\varepsilon$ for the approximation of $\ot_{d_{g}}(\delta_{x}, \delta_{y})$ using a degree $1$ polynomial dual multiplier is given by the explicit formula $\varepsilon = \tfrac{\sin(\theta/2)}{\theta/2}$ when $d_{g}(\delta_{x}, \delta_{y}) = \theta$.
\end{remark}
This clearly identifies the first level of our proposed hierarchy as a simple but \emph{natural} convexification. We argue that this simple convexification is preferable over the naive convexification that one may obtain via a Euclidean parametrization of the manifold $\Omega$.

In \cite[Definition 4.8]{BauermeisterLaudeMollenhoffLiftingLagrangian} the concept of an extremal moment curve is defined. Taking the mapping $\varphi$ in this definition to be the identity map, the unit circle (more generally the unit sphere $\mS{m}$) is an extremal moment curve. Then \cite[Theorem 4.11]{BauermeisterLaudeMollenhoffLiftingLagrangian} proves that any lower semi-continuous cost on $\Omega$ is preserved under our convexification. This is depicted in \cref{fig:convexification}.
\section{Semi-definite program using sum-of-squares certificates} \label{sec:semidefinite}
In this section we finally construct an implementable algorithm to optimize manifold valued MRF problems. We will write our problem as a semi-definite optimization problem where the positive semi-definiteness constraints encode polynomial nonnegativity certificates. For the case of \cref{eq:local_marginal_polytope_dual_metric_discrete}, we rewrite the geodesic Lipschitz continuity condition as a box constraint with respect to the manifold gradient. We then derive a polynomial inequality that is equivalent to this gradient condition, hence implementable in the \emph{semidefinite programming} (SDP) framework. Then we write down the final SDP relaxation hierarchy.
\subsection{Evaluation of the support function}
\label{ssec:SDPdata}
In both the general \cref{eq:local_marginal_polytope_dual_discrete} and metric relaxation \cref{eq:local_marginal_polytope_dual_metric_discrete}, a support function over the space of probability measures needs to be evaluated. This requires optimizing over the space of measures. We assume for the rest of this work that the unary terms are polynomial. For any polynomial $p \in \PE{x}$, we can evaluate the the negative support function 
\begin{equation*}
-\sigma_{\cP(\Omega)}(p) = \min_{\mu \in \cP(\Omega)} \langle p, \mu \rangle = \min_{x \in \Omega} p(x),
\end{equation*}
where the last equality follows from \cref{eq:global_marginal_polytope}, by evaluating the dual problem
\begin{equation*}
\begin{aligned}
\sup_{\zeta \in \bR} \quad & \zeta \\
\text{subject to} \quad& p(x) \geq \zeta \quad \forall x \in \Omega.
\end{aligned}
\end{equation*}
This is a one-dimensional optimization problem, but with a semi-infinite constraint. Inspired by the literature on semidefinite programming relaxations for polynomial optimization, we will relax this to a finite-dimensional \emph{positive semidefinite} (PSD) constraint in \cref{ssec:SDPformulation}.
\subsection{Implementation of the manifold gradient condition}
\label{ssec:SDPregularization}
In order to implement an algorithm to solve the metric relaxation \cref{eq:local_marginal_polytope_dual_metric_discrete}, we require a certificate that ensures Lipschitz continuity of a polynomial. A continuously differentiable function $\varphi \in \cC(\Omega)$ is Lipschitz continuous with respect to the Riemannian metric with $\lip \varphi = L$, if and only if the norm of the manifold gradient of $\varphi$ is bounded by $L$ \cite[Proposition 10.43]{boumal2023intromanifolds}; also see \cite{lellmann2013total}:
\begin{equation} \label{eq:gradient_condition}
\| \gnabla \varphi_{n} \| \leq L,
\end{equation}
where $\gnabla$ is used to denote the manifold gradient.\\
\\
Let $m$ and $M$ be the dimensions of the Riemannian submanifold $\Omega$ and its Euclidean embedding space $\bE$, respectively. There exists an operator $\Pi(x) \in \bR^{M \times M}$, which we assume to be elementwise polynomial, that projects a vector $v \in \bE$ at the point $x \in \Omega$ to the tangent vector $\Pi(x) v \in \TM \Omega(x)$, albeit expressed in the basis of the Euclidean embedding space. Let $\varphi \in \cC(\Omega)$, then the tangent vector $\Pi(x) \nabla \varphi(x)$ equals the manifold gradient of $\varphi$ at $x \in \Omega$. Since $\Pi(x)$ is an orthogonal projection matrix, it is symmetric, idempotent and its eigenvalues are all either $0$ or $1$. If there exists an elementwise polynomial $J(x) \in \bR^{M \times m}$, such that the spectral decomposition of $\Pi(x)$ can be written
\begin{equation*}
\Pi(x) = J(x) J^{*}(x) \quad \forall x \in \Omega,
\end{equation*}
then the columns of $J(x)$ form an orthonormal basis of $\TM \Omega(x)$. Equivalently, for any $v, w \in \bR^{M}$
\begin{equation} \label{eq:projection_matrix_and_Orthonormal_basis}
\langle \Pi(x) v, \Pi(x) w \rangle = \langle \Pi(x) v, w \rangle = \langle J^{*}(x) v, J^{*}(x) w \rangle,
\end{equation}
where we have used the symmetry and idempotency of $\Pi(x)$. This shows that the metric tensor in the inner product between vectors $J^{*}(x)v$ and $J^{*}(x)w$ is the identity, hence $J(x)$ maps vectors $v, w \in \bE$ into an $m$-dimensional vector space with an orthonormal basis. This means in particular that $J^{*}(x)v$ represents $\Pi(x) v$ in an orthonormal basis of $\TM \Omega(x)$.
\begin{remark}[hairy ball theorem]
No such $J(x)$ can exist in the case that $\Omega = \mS{2m}$ for $m \in \bN$; otherwise, one could take any column $j(x)$ of $J(x)$, which would be a polynomial (hence continuous) vector field everywhere tangent to $\mS{2m}$ with $\|j(x)\| = 1$. By the hairy ball theorem, no such nonvanishing continuous tangent vector field can exist.
\end{remark}
In this case however, there may still exist a $J(x) \in \bR^{M \times p}$ for some $p \geq m$, and \cref{eq:projection_matrix_and_Orthonormal_basis} still holds. Furthermore, if $v \in \TM\Omega(x)$ (expressed in the basis of $\bE$), then $v = \Pi(x) v$ and
\begin{equation*}
\| v \|^{2} = \langle v, \Pi(x) v \rangle = \| J^{*}(x) v \|^{2}
\end{equation*}
regardless of whether $p = m$ or not.
\\
\\
Assuming that an elementwise polynomial $J(x) \in \bR^{M \times p}$ exists such that $\Pi(x) = J(x) J^{*}(x)$, then the gradient condition can be implemented via the dual norm. Since
\begin{equation*}
\| \gnabla \varphi(x) \| = \| J^{*}(x) \nabla \varphi(x) \|,
\end{equation*}
(which comes from \cref{eq:projection_matrix_and_Orthonormal_basis} when $v = w = \nabla \varphi$) we have via dualization that
\begin{equation*}
\sup_{\xi \in \bR^{m}} \langle \xi, J^{*}(x) \nabla \varphi(x) \rangle \leq L
\end{equation*}
is equivalent to \cref{eq:gradient_condition}. The gradient condition is met whenever
\begin{equation} \label{eq:manifold_gradient_inequality}
p(x, \xi) = L - \langle \xi, J^{*}(x) \nabla \varphi(x) \rangle \geq 0, \quad \forall \xi \in \mS{p}. 
\end{equation}
This can be implemented as a polynomial inequality in both the variables $x$ and $\xi$, and the coefficient vector of $p$ is linearly dependent on that of $\varphi$. The main advantages of using the matrix $J^{*}(x)$ rather than $\Pi(x)$ in \cref{eq:manifold_gradient_inequality} is that the degree of the polynomial $J^{*}(x) \nabla \varphi(x)$ will be lower than that of $\Pi(x) \nabla \varphi(x)$, as well as requiring a $p$-dimensional dual variable $\xi$, rather than an $M$-dimensional one, noting that $p \leq M$. Both of these facts will result in smaller PSD constraints in our final implementation.
\subsection{Semi-definite problem formulation}
\label{ssec:SDPformulation}
Using the ideas from \cref{ssec:SDPdata} and \cref{ssec:SDPregularization}, the infinite dimensional problems \cref{eq:local_marginal_polytope_dual_discrete} and \cref{eq:local_marginal_polytope_dual_metric_discrete} can be reframed as finite-dimensional optimization problems with semi-infinite polynomial positivity constraints. It is already a classical idea to rephrase polynomial optimization problems, using various Positivstellens\"{a}tze \cite{PutinarPositivstellensatz, StenglePositivstellensatz, KrivinePositivstellensatz}, into large optimization hierarchies \cite{HenrionLasserreBook, BlekhermanParrilloBook}. The idea is to replace polynomial positivity constraints (which are generally intractable to verify) by a positivity certificate that is easy to compute. Following these works, we use Putinar's Postivstellensatz to relax the semi-infinite constraints as the intersection of a semidefinite cone and an affine space. The resulting problem is a finite-dimensional convex semidefinite programming problem. We do not discuss the details of these ideas here, but rather refer to the relevant literature, e.g. the references \cite{HenrionLasserreBook, BlekhermanParrilloBook}. We simply add the following remarks:
\begin{remark}[ball constraint]
Often the additional inequality constraint $q_{I+1}(x) = R^{2} - \|x\|^{2} \leq 0$ for some $R>0$ is introduced in the polynomial hierarchy in order to ensure tightness of the relaxation. The defining equation $\|x\|^2 = 1$ of the unit sphere renders this additional PSD constraint trivially true for any $R \geq 1$, such that it does not need to be implemented.
\end{remark}
\begin{remark}[coordinate ring]
It is possible to reduce the size of the problem by working in the coordinate ring of the maximal ideal that defines the manifold $\Omega$ as its variety. Especially for $m$-dimensional manifolds embedded in an $M$-dimensional space where $m \ll M$ this approach leads to a dramatic reduction. Even for the case $\mS{M-1}$ we consider in this work, however, this dimensionality reduction is significant for dual multipliers with higher degrees.
\end{remark}
\begin{remark}[semidefinite relaxation and moments]
\label{rmk:pseudomomentvectors}
The approach we have taken in this work follows the style of \cite{BlekhermanParrilloBook}. Dual to this approach is the \emph{Lasserre moment hierarchy} \cite{HenrionLasserreBook}, in which one optimizes over (pseudo)moment vectors. In our implementation of this relaxation we compute these (pseudo)moment vectors in order to extract a feasible primal point, cf. \cref{ssec:Rounding}.
\end{remark}
One can prove that the semi-definite program that results from applying the polynomial optimization hierarchy to solve the problems \cref{eq:local_marginal_polytope_dual_discrete} and \cref{eq:local_marginal_polytope_dual_metric_discrete} converges to \cref{eq:local_marginal_polytope} as the degree $n$ of the relaxation and the degrees of the SOS-multipliers are increased.
\section{Experimental evaluation} \label{sec:experiments}

\subsection{Setup and rouding scheme}
\label{ssec:Rounding}
In the subsequent sections, we use our semidefinite relaxation in order to solve some denoising problems with geodesic regularization, an inpainting problem and the pose graph optimization problem. As mentioned in \cref{rmk:pseudomomentvectors}, the optimizers obtained by our relaxation are (pesudo)moment vectors. The degree $1$ moments are extracted and projected onto $\Omega$ to obtain a point $x^{*}$ which is feasible for the primal problem. The duality gap $G$ is computed as the difference $G = P - D$, where $P$ is the cost obtained by $x^{*}$ for the unlifted problem, and $D$ is the optimal value of our semidefinite relaxation. The relative duality gap $G_{R}$ is defined as the ratio between the duality gap and the primal cost: $G_{R} = \tfrac{G}{P}$.

\subsection{$\mS{1}$ denoising problem}
\label{ssec:exp_GD_1}

\begin{figure}[t!]
\centering
\begin{subfigure}[b]{0.3\textwidth}
\centering
\includegraphics[width=\textwidth]{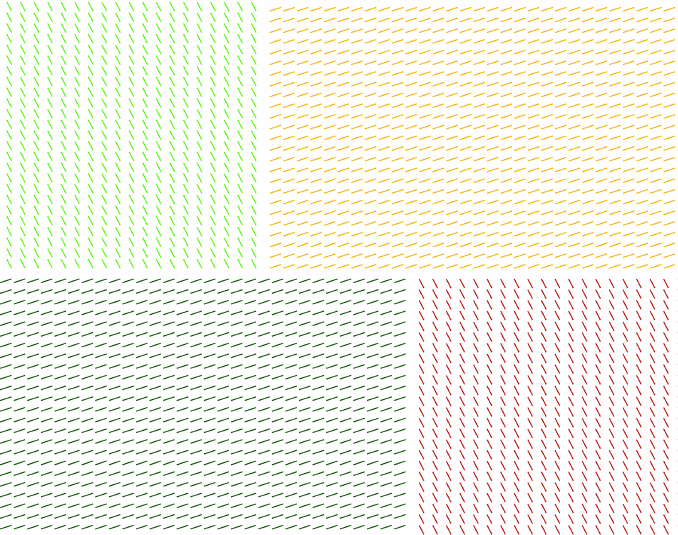}
\end{subfigure}
\begin{subfigure}[b]{0.3\textwidth}
\centering
\includegraphics[width=\textwidth]{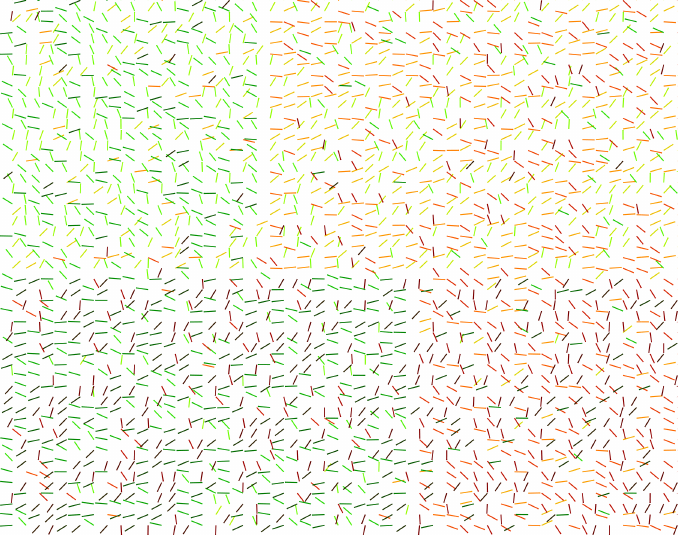}
\end{subfigure}
\begin{subfigure}[b]{0.3\textwidth}
\centering
\includegraphics[width=\textwidth]{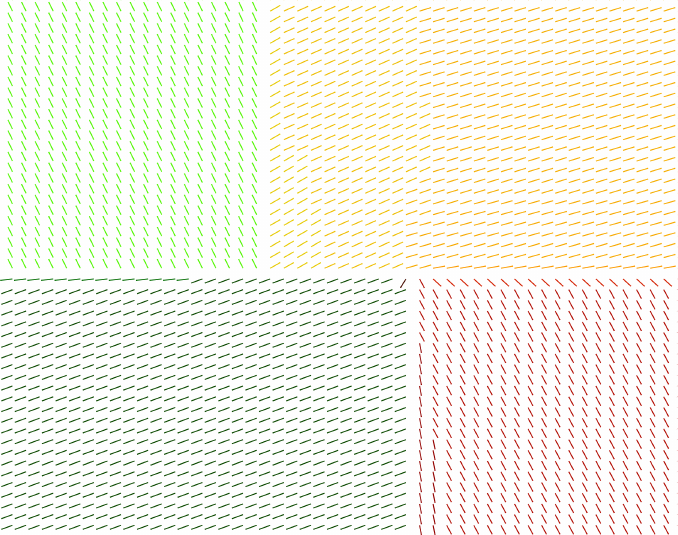}
\end{subfigure}
\begin{subfigure}[b]{0.3\textwidth}
\centering
\includegraphics[width=\textwidth]{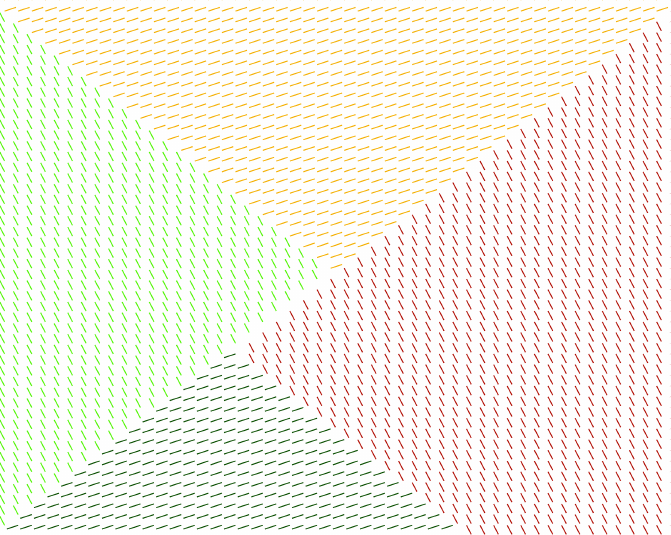}
\caption{}
\end{subfigure}
\begin{subfigure}[b]{0.3\textwidth}
\centering
\includegraphics[width=\textwidth]{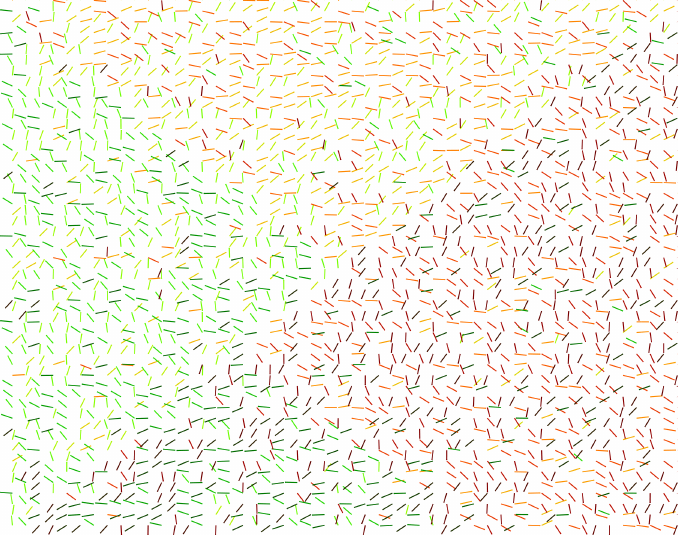}
\caption{}
\end{subfigure}
\begin{subfigure}[b]{0.3\textwidth}
\centering
\includegraphics[width=\textwidth]{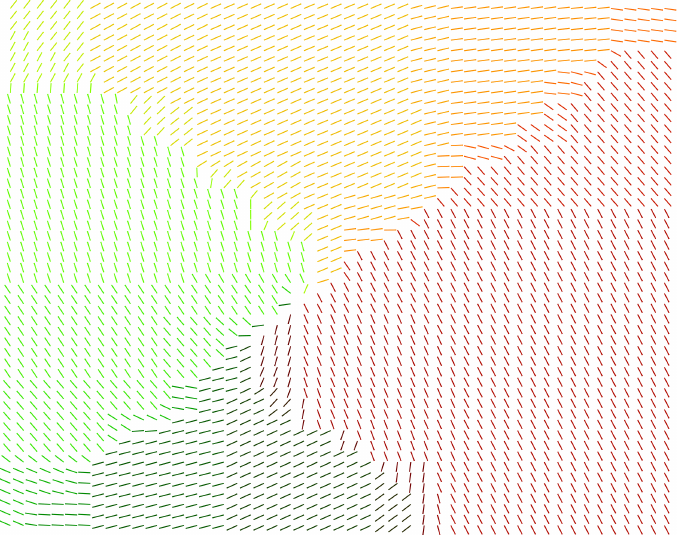}
\caption{}
\end{subfigure}
\caption{Two $\mS{1}$-valued denoising problems (horizontal/vertical and diagonal borders). The ground truth (a) is perturbed by additive Gaussian noise and projected back onto $\mS{1}$ to obtain (b). It is then denoised using our convex relaxation to obtain the denoised (c).}
\label{fig:exp_TOY_denoising}
\end{figure}

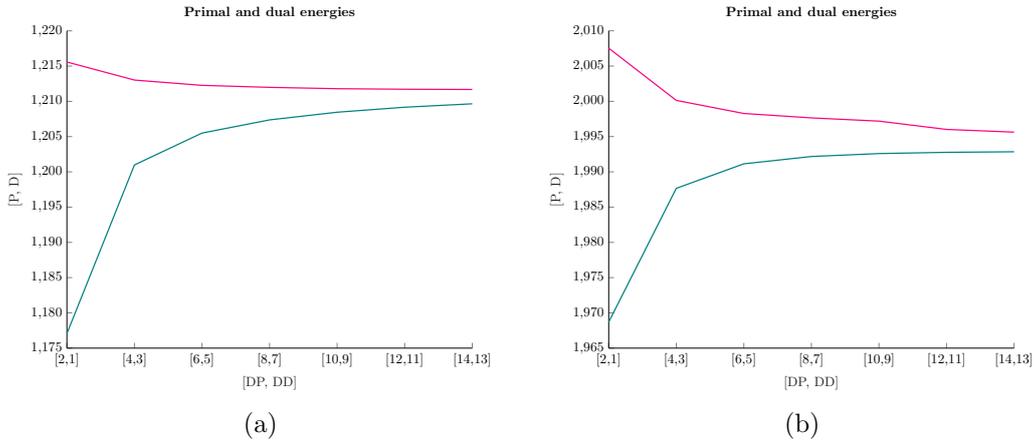
\begin{figure}[t!]
\centering
\begin{subfigure}[b]{0.45\textwidth}
\resizebox {\textwidth} {!} {
\definecolor{mycolor1}{rgb}{1.00000,0.00000,0.50000}%
\definecolor{mycolor2}{rgb}{0.00000,0.50000,0.50000}%
\begin{tikzpicture}

\begin{axis}[%
width=4.521in,
height=3.566in,
at={(0.758in,0.481in)},
scale only axis,
xmin=1,
xmax=7,
xtick={1,2,3,4,5,6,7},
xticklabels={{[2,1]},{[4,3]},{[6,5]},{[8,7]},{[10,9]},{[12,11]},{[14,13]}},
xlabel style={font=\color{white!15!black}},
xlabel={[DP, DD]},
ymin=1175,
ymax=1220,
ylabel style={font=\color{white!15!black}},
ylabel={[P, D]},
axis background/.style={fill=white},
title style={font=\bfseries},
title={Primal and dual energies},
axis x line*=bottom,
axis y line*=left
]
\addplot [color=mycolor1, line width=1.0pt, forget plot]
  table[row sep=crcr]{%
1	1215.57694320791\\
2	1213.01560783267\\
3	1212.26749446578\\
4	1211.98223175325\\
5	1211.79397856748\\
6	1211.71496840485\\
7	1211.68250197372\\
};
\addplot [color=mycolor2, line width=1.0pt, forget plot]
  table[row sep=crcr]{%
1	1177.06972430263\\
2	1200.98343661496\\
3	1205.49951298774\\
4	1207.36933409947\\
5	1208.45578133141\\
6	1209.16697921574\\
7	1209.65034347859\\
};
\end{axis}

\begin{axis}[%
width=5.833in,
height=4.375in,
at={(0in,0in)},
scale only axis,
xmin=0,
xmax=1,
ymin=0,
ymax=1,
axis line style={draw=none},
ticks=none,
axis x line*=bottom,
axis y line*=left
]
\end{axis}
\end{tikzpicture}%
}
\caption{}
\end{subfigure}
\centering
\begin{subfigure}[b]{0.45\textwidth}
\resizebox {\textwidth} {!} {
\definecolor{mycolor1}{rgb}{1.00000,0.00000,0.50000}%
\definecolor{mycolor2}{rgb}{0.00000,0.50000,0.50000}%
\begin{tikzpicture}

\begin{axis}[%
width=4.521in,
height=3.566in,
at={(0.758in,0.481in)},
scale only axis,
xmin=1,
xmax=7,
xtick={1,2,3,4,5,6,7},
xticklabels={{[2,1]},{[4,3]},{[6,5]},{[8,7]},{[10,9]},{[12,11]},{[14,13]}},
xlabel style={font=\color{white!15!black}},
xlabel={[DP, DD]},
ymin=1965,
ymax=2010,
ylabel style={font=\color{white!15!black}},
ylabel={[P, D]},
axis background/.style={fill=white},
title style={font=\bfseries},
title={Primal and dual energies},
axis x line*=bottom,
axis y line*=left
]
\addplot [color=mycolor1, line width=1.0pt, forget plot]
  table[row sep=crcr]{%
1	2007.53776686569\\
2	2000.13289575322\\
3	1998.27842311619\\
4	1997.65066928545\\
5	1997.19450808206\\
6	1996.01817833309\\
7	1995.63119387394\\
};
\addplot [color=mycolor2, line width=1.0pt, forget plot]
  table[row sep=crcr]{%
1	1968.69354813277\\
2	1987.66363405519\\
3	1991.13763830334\\
4	1992.17687131155\\
5	1992.592416876\\
6	1992.76941623547\\
7	1992.84825996464\\
};
\end{axis}

\begin{axis}[%
width=5.833in,
height=4.375in,
at={(0in,0in)},
scale only axis,
xmin=0,
xmax=1,
ymin=0,
ymax=1,
axis line style={draw=none},
ticks=none,
axis x line*=bottom,
axis y line*=left
]
\end{axis}
\end{tikzpicture}%
}
\caption{}
\end{subfigure}
\caption{The gap between the primal and dual energies for increasing degrees of the relaxation for a circle-valued denoising problem (a) horizontal/vertical borders (b) diagonal borders. The tuple $[DP, DD]$ on the horizontal axis are the degrees at which the primal (pseudo)moment vectors and the dual polynomial multipliers are truncated, respectively.}
\label{fig:exp_TOY_PD_energies}
\end{figure}

In the following experiments, we consider the geodesic denoising problem:
\begin{equation}
\label{eq:geodesic_denoising_cost}
T_{\text{denoised}} = \argmin_{Y \in (\Omega)^{\cV}} \sum_{u \in \cV} \tfrac{(T_{u} - Y_{u})^{2}}{2} + L \sum_{(u,v) \in \cE} d_{g}(Y_{u}, Y_{v})
\end{equation}
where $Y$ is used to denote the given dataset. For this experiment we have $\Omega = \mS{1}$.\\
\\
We construct a toy dataset which consists of a $\mS{1}$-valued image with piecewise constant regions seperated by straight borders. In contrast to the continuous models studied in e.g. \cite{LellmannSchorrContinuousLabeling}, using the graph-based MRF-model of the image as studied here leads to anisotropic denoising \cite{BoykovKolmogorovGeodesics}. This is made clear by considering the results of the experiment with horizontal/vertical borders, for which the denoising result is very close to the ground truth, versus the experiment with diagonal borders, where a clear staircasing effect occurs at the borders. We plot the noiseless, noisy and denoised dataset in \cref{fig:exp_TOY_denoising}, as well as the duality gap for various degrees of the relaxation in \cref{fig:exp_TOY_PD_energies}.

\subsection{Vesuvius SAR denoising problem}
\label{ssec:exp_GD_2}
\begin{figure}[t!]
\centering
\begin{subfigure}[b]{0.23\textwidth}
\centering
\includegraphics[width=\textwidth]{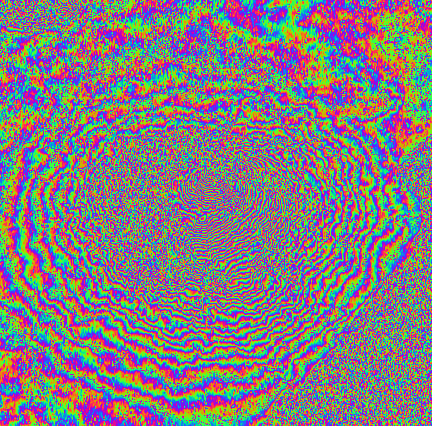}
\caption{}
\end{subfigure}
\begin{subfigure}[b]{0.23\textwidth}
\centering
\includegraphics[width=\textwidth]{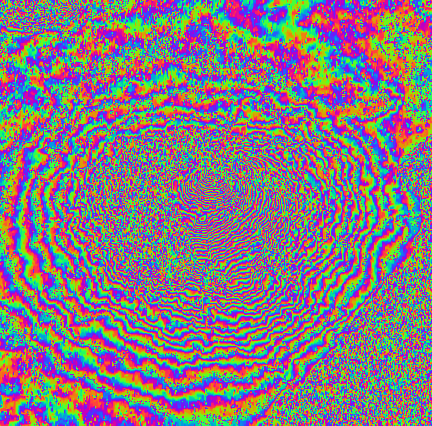}
\caption{}
\end{subfigure}
\begin{subfigure}[b]{0.23\textwidth}
\centering
\includegraphics[width=\textwidth]{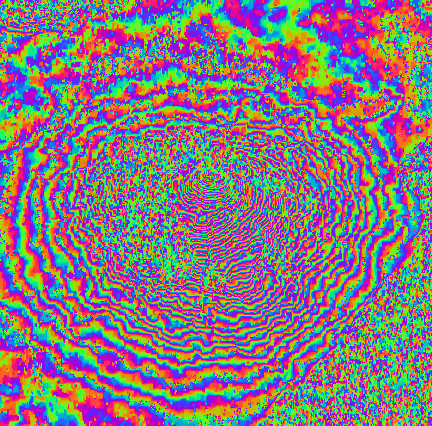}
\caption{}
\end{subfigure}
\begin{subfigure}[b]{0.23\textwidth}
\centering
\includegraphics[width=\textwidth]{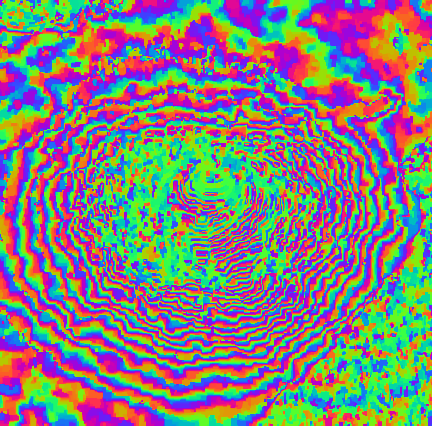}
\caption{}
\end{subfigure}
\caption{Denoising of Vesuvius INSAR data, dataset from \cite{RoccaPratiFerrettiSAR}: input (a), denoising with geodesic regularization parameter $L=0.1$ (b), $L=0.25$ (c) and $L=0.5$ (d).}
\label{fig:exp_SAR_denoising}
\end{figure}

\begin{table}
\centering
\begin{tabular}{ |c|c|c|c|c| } 
 \hline
  & primal ($P$) & dual ($D$) & gap ($G$) & relative gap ($G_{R}$) \\ \hline
 $L = 0.1$ & $2442.2533$ & $2350.37250$ & $918.80809$ & $3.7621 \times 10^{-2}$ \\
 $L = 0.25$ & $5249.2556$ & $5111.4064$ & $137.8491$ & $2.6261 \times 10^{-2}$ \\
 $L = 0.5$ & $7911.2901$ & $7826.6684$ & $846.2167$ & $1.0696 \times 10^{-2}$\\ \hline
\end{tabular}
\caption{Primal and dual energies for the Vesuvius denoising problems. The experiments were carried out using degree $3$ polynomial dual variables.}
\label{tab:exp_SAR_denoising}
\end{table}

In InSAR (Interferometric Synthetic Aperture Radar) imaging, information about the deformation of a terrain is obtained in the form of phase measurements that reflect the phase shift between outgoing and returning radar pulses. Phase measurements are naturally interpreted as elements of $\mS{1}$, which makes denoising of an InSAR image a natural application of our framework. The dataset is denoised for various levels of geodesic regularization, using the cost function given in \cref{eq:geodesic_denoising_cost}, with the results shown in \cref{fig:exp_SAR_denoising} and \cref{tab:exp_SAR_denoising}. The degree of the dual multipliers in our relaxation (in accordance to the dual viewpoint \cref{eq:local_marginal_polytope_dual_metric_discrete}) was restriced to $3$.

\subsection{ROF-denoising and inpainting problem}
\label{ssec:exp_GD_3}
\begin{figure}[t!]
\centering
\begin{subfigure}[b]{0.3\textwidth}
\centering
\includegraphics[width=\textwidth]{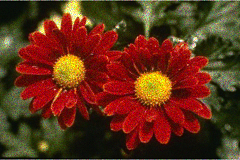}
\end{subfigure}
\begin{subfigure}[b]{0.3\textwidth}
\centering
\includegraphics[width=\textwidth]{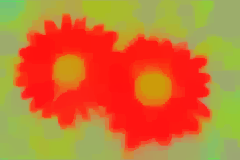}
\end{subfigure}
\begin{subfigure}[b]{0.3\textwidth}
\centering
\includegraphics[width=\textwidth]{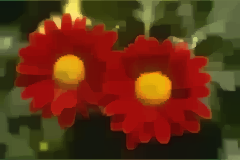}
\end{subfigure}
\\
\begin{subfigure}[b]{0.3\textwidth}
\centering
\includegraphics[width=\textwidth]{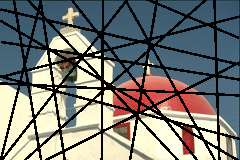}
\caption{}
\end{subfigure}
\begin{subfigure}[b]{0.3\textwidth}
\centering
\includegraphics[width=\textwidth]{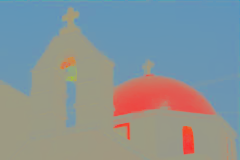}
\caption{}
\end{subfigure}
\begin{subfigure}[b]{0.3\textwidth}
\centering
\includegraphics[width=\textwidth]{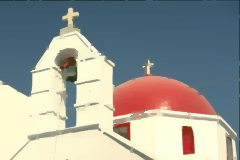}
\caption{}
\end{subfigure}
\caption{ROF-Denoising (row 1) and inpainting (row 2) of an image by optimization over intensity and chromaticity separately: input data (a), optimal chromaticity (b) and output (c). The images were taken from \cite{MartinFTM01}.}
\label{fig:exp_ROF_denoising}
\end{figure}

The well-known TV-denoising or \emph{Rudin-Osher-Fatemi} (ROF) model, introduced in \cite{RudinOsherFatemiDenoising}, is a classical method for denoising images. This model is exactly that given in \cref{eq:geodesic_denoising_cost}, when $\Omega = \bR^{3}$, and $d(x, y) = \|x - y\|$. In \cite{KangMarchROFdenoising} a similar method is proposed, where rather than treating the pixels of an image $T$ as elements in $\bR^{3}$, every pixel $x$ is decomposed into its intensity $i_{x}$ and chromaticity $c_{x}$ as $(i_{x}, c_{x}) = (|x|, \tfrac{x}{|x|}) \in \bR \times \mS{2}$. In essence, one is solving \cref{eq:geodesic_denoising_cost} for the intensity and chromaticity separately, where $\Omega = \bR$ and $\Omega = \mS{2}$, respectively. The former can be computed using any number of convex optimization techniques; we use our framework in order to compute the latter. This model can also be used to solve the inpainting problem, where for some pixels the image is corrupted. We again solve \cref{eq:geodesic_denoising_cost}, where for the missing pixels the dateterm is set to zero. For both experiments, the degree of the dual multipliers (in the dual formulation \cref{eq:local_marginal_polytope_dual_metric_discrete}) was restricted to $1$, which means we are implicitely solving the convexification \cref{eq:degree_1_relaxation_metric} studied in \cref{ssec:degree_1_case}. The noisy input, the denoised chromaticity and the denoised output images are shown in \cref{fig:exp_ROF_denoising}, and the primal and dual energies are shown in \cref{tab:exp_ROF_denoising}.
\begin{table}
\centering
\begin{tabular}{ |c|c|c|c|c| } 
 \hline
  & primal ($P$) & dual ($D$) & gap ($G$) & relative gap ($G_{R}$) \\ \hline
 Flowers & $766.2644$ & $766.1632$ & $0.1012$ & $1.3210 \times 10^{-4}$ \\
 Rooftop & $44.21310$ & $44.15262$ & $0.060486$ & $1.3681 \times 10^{-3}$ \\ \hline
\end{tabular}
\caption{Primal and dual energies for the image denoising problems. The experiments were carried out for the degree $1$-relaxation, as studied in \cref{ssec:degree_1_case}.}
\label{tab:exp_ROF_denoising}
\end{table}

\subsection{Pose graph optimization}
\label{ssec:exp_PGO}
\begin{figure}[t!]
\centering
\begin{subfigure}[b]{0.2\textwidth}
\centering
\includegraphics[width=\textwidth]{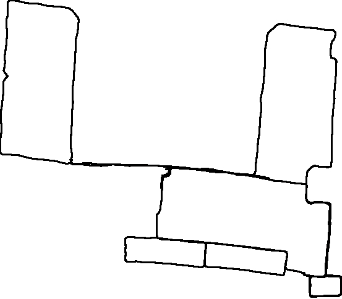}
\end{subfigure}
\hfill
\begin{subfigure}[b]{0.2\textwidth}
\centering
\includegraphics[width=\textwidth]{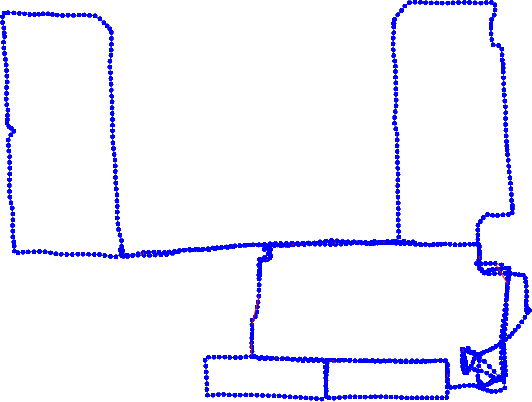}
\end{subfigure}
\hfill
\begin{subfigure}[b]{0.2\textwidth}
\centering
\includegraphics[width=\textwidth]{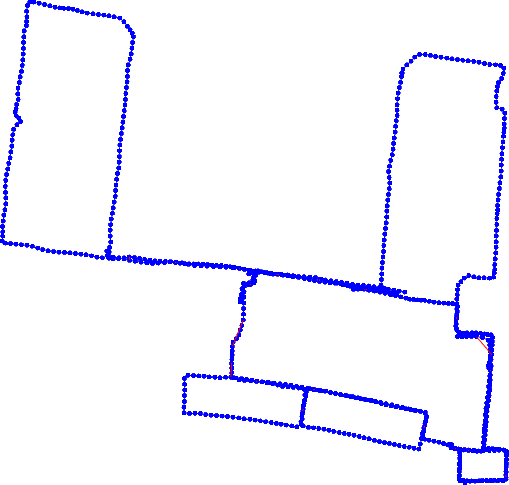}
\end{subfigure}
\hfill
\begin{subfigure}[b]{0.2\textwidth}
\centering
\includegraphics[width=\textwidth]{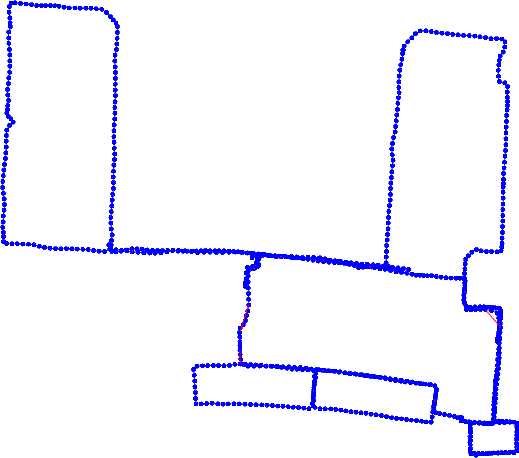}
\end{subfigure}
\\
\begin{subfigure}[b]{0.2\textwidth}
\centering
\includegraphics[width=\textwidth]{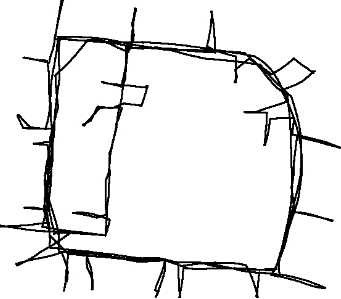}
\end{subfigure}
\hfill
\begin{subfigure}[b]{0.2\textwidth}
\centering
\includegraphics[width=\textwidth]{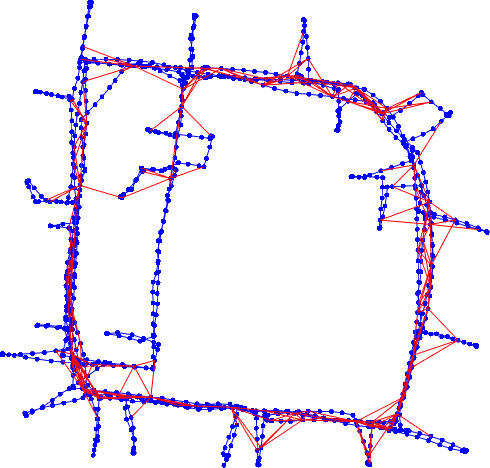}
\end{subfigure}
\hfill
\begin{subfigure}[b]{0.2\textwidth}
\centering
\includegraphics[width=\textwidth]{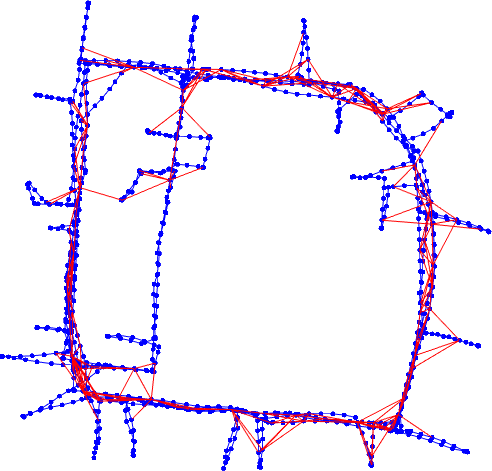}
\end{subfigure}
\hfill
\begin{subfigure}[b]{0.2\textwidth}
\centering
\includegraphics[width=\textwidth]{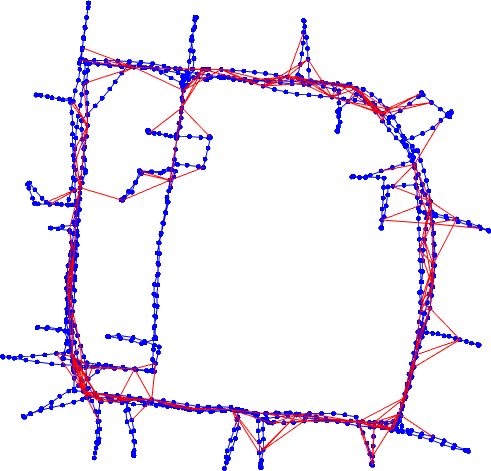}
\end{subfigure}
\\
\begin{subfigure}[b]{0.2\textwidth}
\centering
\includegraphics[width=\textwidth]{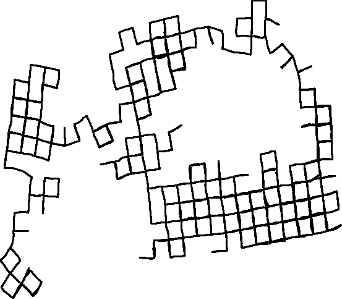}
\caption{}
\end{subfigure}
\hfill
\begin{subfigure}[b]{0.2\textwidth}
\centering
\includegraphics[width=\textwidth]{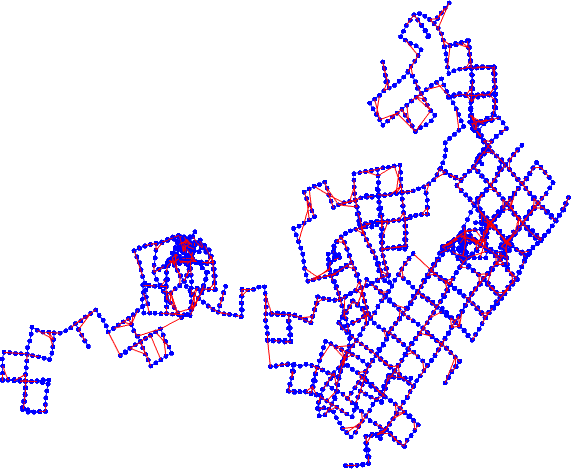}
\caption{}
\end{subfigure}
\hfill
\begin{subfigure}[b]{0.2\textwidth}
\centering
\includegraphics[width=\textwidth]{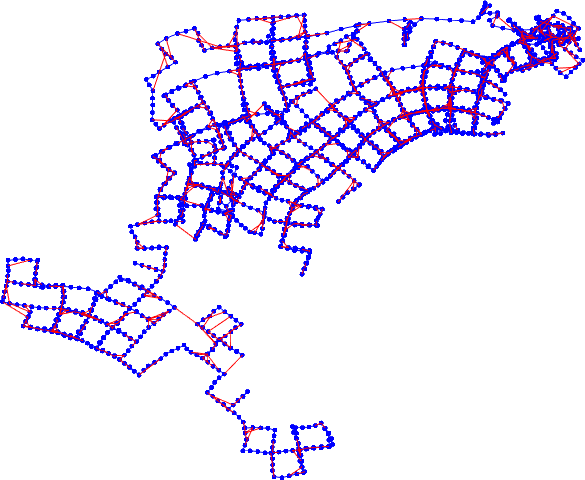}
\caption{}
\end{subfigure}
\hfill
\begin{subfigure}[b]{0.2\textwidth}
\centering
\includegraphics[width=\textwidth]{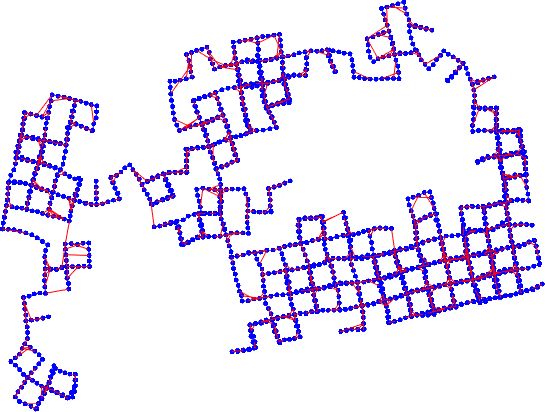}
\caption{}
\end{subfigure}
\caption{Results for the pose graph optimization problem: ground truth (a), local method (b), ours with least squares rounding (c), ours after local optimization with ManOpt \cite{manopt} (d). The dataset was taken from \cite{CarloneIntegerLattice}.}
\label{fig:exp_PGO_comparison}
\end{figure}

In \emph{pose graph optimization} (PGO), an agent travels along an unknown path, which is modeled by a sequence of positions the agent has visited. These positions along with the direction the agent was pointing towards at that positions (encoded by a rotation from a reference direction) are named poses, and these form the vertices of the underlying graph for this problem. The agent keeps track of (noisy) measurements of its relative rotations $r$ and displacements $x$ between two subsequent poses. These measurements are naturally elements of the special Euclidean group: $(x, r) \in \mSE{2} \cong \bR^{2} \times \mSO{2}$. Two vertices are connected whenever they are subsequent positions in the path, but also whenever the agent has returned to a previous position, in which case the edge is called a loop closure. The goal of PGO is to estimate the true path of travel based only on noisy relative measurements and the loop closures. The cost function for PGO is \cite{GrisettiKummerleGraphBasedSlam}
\begin{equation} \label{eq:PGO_cost_function}
\min_{(y, q) \in (\mSE{2})^{\cV}} \sum_{(u, v) \in \cE} \tfrac{\kappa_{e}}{2} \big\| y_{v} - y_{u} - q_{u} x_{(u, v)} \big\|^{2}_{2} + \tfrac{\lambda_{e}}{2} \big\| q_{v} - q_{u} r_{(u, v)} \big\|^{2}_{2}
\end{equation}
where $(x, r) \in (\mSE{2})^{\cE}$ contains the noisy relative measurements between positions, including those estimated for the loop closure edges. The solution $(y, q) \in (\mSE{2})^{\cV}$ contains the absolute poses.\\
\\
Semidefinite relaxations for the PGO problem have been studied widely, including from the viewpoint of sparse polynomial programming \cite{CarloneConvexRelaxationsPGO, YangGPO, mangelson2019guaranteed}. We solve the PGO problem using our relaxation, using a polynomial of degree $2$ as dual variable, and either use a least squares optimization step to update the positions while keeping the angles fixed (LSQ rounding), or initialize a local Riemannian optimization method with our solution (Manopt \cite{manopt}) rounding. The results are found in \cref{fig:exp_PGO_comparison}. While a local method is required to refine the solution for the PGO problem, we remark that for the rotation estimation problem (which is obtained by fixing the poses $y$ and optimizing only over the rotations $q$ in \cref{eq:PGO_cost_function}) our relaxation is empirically tight in low-noise regimes even at degree $2$. The appeal of our hierarchy for these problems is that leads to smaller optimization problems than existing semidefinite relaxations.

%\subsection{Rotation averaging}
%
%A subproblem of the pose graph optimization problem for which our framework is  more applicable is the `rotation averaging' problem. We attempt to minimize the cost function 
%\begin{equation} \label{eq:RA_cost_function}
%\min_{w = (y, q) \in (\mSE{2})^{\cV}} \sum_{(u, v) \in \cE} \tfrac{\kappa_{e}}{2} \big\| y_{v} - y_{u} - q_{u} x_{(u, v)} \big\|^{2}_{2} + \tfrac{\lambda_{e}}{2} \big\| q_{v} - q_{u} r_{(u, v)} \big\|^{2}_{2}
%\end{equation}
%for which the relative translations $x_{(u,v)}$ and rotations $r_{(u,v)}$, as well as the absolute translations $y_{u}$ and $y_{v}$ are assumed to be given. This problem is obtained by keeping the absolute translations in the PGO problem \ref{eq:PGO_cost_function} fixed.\\
%\\
%We solve the rotation averaging problem with our relaxation using dual multipliers of degree $1$ and $2$. The relative duality gaps are given below:
%\begin{center}
%\begin{tabular}{| c | c c c c c c |}
%\hline
%dataset & MIT & Intel & M3500 & M3500a & M3500b & M3500c \\ \hline
%degree 1 & $0.8412$ & $0.1781$ & $0.0935$ & $0.0576$ & $0.0717$ & $0.0780$ \\
%degree 2 & $0$ & $0$ & $0$ & $0$ & $0$ & $0$ \\ \hline
%\end{tabular}
%\end{center}

\section{Conclusion}
In this paper we have considered an efficient moment relaxation for optimization problems with graphical structure and spherical constraints. Leveraging optimal transport duality we derive a tractable dual subspace approximation of the infinite-dimensional problem which allows us to tackle possibly nonpolynomial optimization problems with manifold constraints and geodesic coupling terms. We showed that the duality gap vanishes in the limit by proving that a Lipschitz continuous dual multiplier on a unit sphere can be approximated as closely as desired in terms of a Lipschitz continuous polynomial. The formulation is applied to manifold-valued imaging problems with total variation regularization and graph-based SLAM. In imaging tasks our approach achieves small duality gaps for a moderate degree. In graph-based SLAM our approach often finds solutions which after refinement with a local method are near the ground truth solution.
This shows that in contrast to existing approaches our approach scales to large problems with millions of variables.

\printbibliography

\appendix

\section{A counterexample to tightness for the local marginal polytope relaxation}

In \cite{BauermeisterLaudeMollenhoffLiftingLagrangian} the authors prove tightness of the local marginal polytope relaxation for distance-like coupling terms by leveraging results from \cite{BachSubmodular}. We construct a simple example that shows that for any graph containing at least one loop, the local marginal polytope relaxation is not tight, even for spheres under coupling terms penalizing geodesic distance.

\begin{example}[Local marginal polytope relaxation is not tight] \label{ex:counterexample_LMP}
Consider a graph $\cG = (\cV, \cE)$ with $K$ edges, indexed by the integers $0 \ldots K-1$. Assume that $\cE$ contains all edges $(k, k+1)$ for $k = 0 \ldots K-1$. Here, for notational convenience, we have indexed the $0$-th vertex also by $K$ so as to write the edge $(K-1, K)$ rather than $(K-1, 0)$. Furthermore, consider a set of $2K$ equally spaced points distributed around the unit circle, indexed by integers mod $0 \ldots 2K-1$. These points are ordered around the circle in increasing order such that $d_{g}(y_{k}, y_{k+1}) = \tfrac{\pi}{K}$. Then $d_{g}(y_{k}, y_{k+K}) = \pi$ and $d_{g}(y_{k}, y_{k+K+1}) = \tfrac{(K-1) \pi}{K}$. Finally we state the optimization problem, where for every vertex $k$ we optimize a variable $x_{k} \in \mS{1}$:
\begin{equation}
\min_{x \in (\mS{1})^{\cV}} F_{K}(x) = \min_{x \in (\mS{1})^{\cV}} \sum_{k = 0}^{K-1} \ind_{\{y_{k}, y_{K+k}\}}(x_{k}) + \sum_{k = 0}^{K-1} d_{g}(x_{k}, x_{k+1}).
\end{equation}
Consider the pair $(x_{k}, x_{k+1})$. The $k^{\text{th}}$ data term enforces that $x_{k} = y_{k}$ or $x_{k} = y_{k+K}$, and similarly for $x_{k+1}$. Notice that $d_{g}(x_{k}, x_{k+1}) = \tfrac{1}{K}\pi$ if $(x_{k}, x_{k+1}) = (y_{k}, y_{k+1})$ or $(x_{k}, x_{k+1}) = (y_{k+K}, y_{k+K+1})$, and  $d_{g}(x_{k}, x_{k+1}) = \tfrac{K-1}{K}\pi$ otherwise. We prove that $x^{*} = (y_{0}, y_{1}, \ldots, y_{K})$ is an optimizer to this problem.\\
\\
To see this, note that there must be at least one edge $(x_{k}, x_{k+1})$ which achieves the cost $d_{g}(x_{k}, x_{k+1}) = \tfrac{K-1}{K} \pi$. Conversely, assume there exists a solution for which this is not true, and let $x_{0} = y_{l}$ (where $l = 0$ or $l=K$). Since $d_{g}(x_{0}, x_{1}) = \tfrac{1}{K}\pi$, we must have $x_{1} = y_{l+1}$. Repeating this argument for every edge, we obtain $x_{K} = y_{l+K}$, which is impossible since $x_{0} = x_{K}$, and $y_{l} \neq y_{l+K}$. Hence at least one edge achieves cost $d_{g}(x_{k}, x_{k+1}) = \tfrac{K-1}{K} \pi$, and every other edge achieves a cost $d_{g}(x_{k}, x_{k+1}) \geq \tfrac{1}{K}\pi$, which means that $F_{K}(x) \geq 2 \tfrac{K-1}{K} \pi$. This cost is achieved by $x^{*}$ (as well as $2 K$ equivalent solutions obtained by symmetry): $F(x^{*}) = 2 \tfrac{K-1}{K} \pi$.\\
\\
Now consider the local marginal polytope relaxation for this problem:
\begin{equation}
\min_{\mu \in \cM(\mS{1})^{\cV}} \cF(\mu) = \min_{\mu \in \cM(\mS{1})^{\cV}} \sum_{k = 0}^{K-1} \langle \ind_{\{y_{k}, y_{K+k}\}}, \mu_{k} \rangle + \sum_{k = 0}^{K-1} \ot_{d_{g}}(\mu_{k}, \mu_{k+1}).
\end{equation}
We can consider the solution $\mu^{*}$ with $\mu_{k}^{*} = (\tfrac{\delta_{y_{k}} + \delta_{y_{K+k}}}{2})$. This solution turns out to be optimal, but it suffices for this example that it is feasible and achieves lower cost than $\delta_{x^{*}}$. To see this, note that $\mu_{k}^{*}$ is supported only on $\{y_{k}, y_{K+k}\}$, hence $\langle \ind_{\{y_{k}, y_{K+k}\}}, \mu^{*}_{k} \rangle = 0$. Furthermore, we have
\begin{align*}
\ot_{d_{g}}(\mu^{*}_{k}, \mu^{*}_{k+1}) &= \ot_{d_{g}}(\tfrac{\delta_{y_{k}} + \delta_{y_{K+k}}}{2}, \tfrac{\delta_{y_{k+1}} + \delta_{y_{K+k+1}}}{2}) \\ 
&\leq \tfrac{1}{2} \langle d_{g}, \delta_{y_{k}} \times \delta_{y_{k+1}} + \delta_{y_{K+k}} \times \delta_{y_{K+k+1}} \rangle \\
&= \frac{d_{g}(y_{k}, y_{k+1}) + d_{g}(y_{K+k}, y_{K+k+1})}{2} = \tfrac{1}{K} \pi.
\end{align*}
Since this inequality holds for every edge, we have that the cost of the local marginal polytope relaxation is upper bounded by $K \cdot \tfrac{1}{k} \pi = \pi$ (in fact $\cF(\mu^{*}) = \pi$) which is strictly less than the minimal cost $2 \tfrac{K-1}{K} \pi$ obtained for the unlifted problem. Therefore, the local marginal polytope relaxation is not tight.
\end{example}
It is interesting to note that it is not necessary to resort to extended real-valued functions for this result to hold. With some extra effort the same result can be shown for $f_{k}(x) = 2 \pi |x_{1} (y_{k})_{2} - x_{2} (y_{k})_{1}|$, rather than explicitly constraining $x_{k}$ via the indicator functions $f_{k}(x) = \ind_{\{y_{k}, y_{-k}\}}(x)$. The proof is obtained by showing that, for any fixed $x_{k-1}$ and $x_{k+1}$, the optimal value for $x_{k}$ is either $y_{k}$ or $y_{k+K}$. The rest of the proof then follows the proof above. Finally, this example can be generalized to any arbitrary graph that contains at least one loop of size $K$, by simply setting $f_{u} = 0$ for all vertices $u$ that are not part of the loop; the local marginal polytope relaxations is only tight for trees.

\section{Proof of \cref{thm:duality_general}} \label{secsup:duality_general_proof}

\begin{proof} \label{prf:duality_general}
Define $F: \cC(\Omega)^\cV \to \exR$ and $G : (\cC(\Omega)^\cE)^2 \to \exR$ as
\begin{equation}
F(\alpha)=\sum_{u \in \cV}\sigma_{\cP(\Omega)}(\alpha_u - f_u), \qquad \text{and} \qquad G(\varphi, \psi)= \sum_{e\in \cE} \iota_{\cK_e}(\varphi_e, \psi_e).
\end{equation}
Then the problem \cref{eq:local_marginal_polytope_dual} can be rewritten as
\begin{equation*}
\inf_{(\varphi,\psi) \in (\cC(\Omega)^\cE)^2} F(A(\varphi, \psi)) + G(\varphi, \psi).
\end{equation*}
By definition we have $(\iota_{\cP(\Omega)} + \langle f_u, \cdot \rangle)^*=\sigma_{\cP(\Omega)}(\cdot - f_u)$ which is proper, convex and lsc.
In view of closedness of $\cP(\Omega)$ in the weak$^*$ topology and \cite[Lemma 1.6]{santambrogio2015optimal} we have that $\iota_{\cP(\Omega)} + \langle f_u, \cdot \rangle$ is proper, convex and lsc as well implying that %in view of \cite[Equation 2.2]{rockafellar1967duality} 
$(\sigma_{\cP(\Omega)}(\cdot - f_u))^{*} = \iota_{\cP(\Omega)} + \langle f_u, \cdot \rangle$.
Therefore $F^*:\cM(\Omega)^\cV \to \exR$ with $F^*(\mu)=\sum_{u \in \cV} \iota_{\cP(\Omega)}(\mu) + \langle f_u, \mu_u \rangle$.
We have $G^*: (\cM(\Omega)^\cE)^2 \to \exR$ with
\begin{equation*}
G^*(\mu, \nu)=\sum_{e \in \cE} \sigma_{\cK_e}(\mu, \nu)=\sup_{(\varphi,\psi) \in (\cC(\Omega)^\cE)^2} \langle \varphi, \mu \rangle + \langle \psi, \nu \rangle - \sum_{e \in \cE}\iota_{\cK_e}(\varphi_e,\psi_e).
\end{equation*}
Then the dual problem amounts to
\begin{align}
\sup_{\mu \in \cP(\Omega)^\cV} -G^*(-A^*(\mu)) -\sum_{u \in \cV} \langle f_u, \mu_u \rangle,
\end{align}
for 
\begin{align*}
G^*(-A^*(\mu)) &= \sup_{\varphi,\psi \in \cC(\Omega)^\cE} -\langle (\varphi,\psi), A^*(\mu) \rangle - \sum_{e \in \cE}\iota_{\cK_e}(\varphi_e,\psi_e) \\
&=\sup_{\varphi,\psi \in \cC(\Omega)^\cE}-\langle A(\varphi,\psi), \mu \rangle - \sum_{e \in \cE}\iota_{\cK_e}(\varphi_e,\psi_e) \\
&=\sup_{\varphi,\psi \in \cC(\Omega)^\cE}\sum_{u \in \cV} \langle -A_u(\varphi,\psi), \mu_u \rangle - \sum_{e \in \cE}\iota_{\cK_e}(\varphi_e,\psi_e) \\
&= \sup_{\varphi,\psi \in \cC(\Omega)^\cE}\sum_{u \in \cV} \sum_{v :(u,v) \in \cE} \langle \varphi_{(u,v)},\mu_u\rangle +  \sum_{v \in \cV}\sum_{u :(u,v) \in \cE} \langle \psi_{(u,v)}, \mu_v \rangle - \sum_{e \in \cE}\iota_{\cK_e}(\varphi_e,\psi_e) \\
&=\sup_{\varphi,\psi \in \cC(\Omega)^\cE}\sum_{(u,v) \in \cE} \langle \varphi_{(u,v)},\mu_u \rangle + \langle \psi_{(u,v)}, \mu_v \rangle - \iota_{\cK_{(u,v)}}(\varphi_{(u,v)},\psi_{(u,v)}) \\
&=\sum_{(u,v) \in \cE} \sup_{\varphi,\psi \in \cC(\Omega)}\langle \varphi,\mu_u \rangle + \langle \psi, \mu_v \rangle - \iota_{\cK_{(u,v)}}(\varphi,\psi) \\
&=\sum_{(u,v) \in \cE} \mathrm{OT}_{f_{(u,v)}}(\mu_u, \mu_v),
\end{align*}
where the last equality follows from \cite[Theorem 1.42]{santambrogio2015optimal} and the fact that $\Omega$ is compact and $f_{(u,v)}$ lsc.

Next we are going to verify that $\inf_{\varphi,\psi \in \cC(\Omega)^\cE} F(A(\varphi, \psi)) + G(\varphi, \psi)$ satisfies the conditions in \cite[Theorem 1]{rockafellar1967duality}:
We have already seen that $F$ is proper, convex and lsc since it is the convex conjugate of a proper, lsc convex function.

Consider a sequence $(\varphi_e^k, \psi_e^k) \in \cK_e$ that converges uniformly to $(\varphi_e,\psi_e)$, i.e., $(\varphi_e^k, \psi_e^k) \to (\varphi_e,\psi_e)$ for $k \to \infty$. Then we have
\begin{equation*}
-\varphi_e^k(x) - \psi_e^k(y) +f_e(x,y) \geq 0 \quad \forall x,y \in \Omega \times \Omega,
\end{equation*}
for any $k \in \bN$.
Let $x,y \in \Omega \times \Omega$ arbitrary but fixed.
Adding $\varphi_e(x) + \psi_e(y)$ to both sides of the inequality this implies that
\begin{equation*}
\varphi_e(x)-\varphi_e^k(x) + \psi_e(y) - \psi_e^k(y) +f_e(x,y) \geq \varphi_e(x) + \psi_e(y).
\end{equation*}
Passing $k \to \infty$ we obtain since uniform convergence implies pointwise convergence that $\varphi_e(x)-\varphi_e^k(x) \to 0$ and $\psi_e(y) - \psi_e^k(y) \to 0$. This implies that
\begin{equation}
f_e(x,y) \geq \varphi_e(x) + \psi_e(y).
\end{equation}
Since $x,y \in \Omega$ are arbitrary this implies that $(\varphi_e, \psi_e) \in \cK_e$ and thus $\cK_e$ is closed implying that $G$ is lsc. Clearly $G$ is also convex and proper.

Note that $A$ is linear and continuous as a weighted sum of $(\varphi,\psi) \in \cC(\Omega)^\cE \times \cC(\Omega)^\cE$.
Without loss of generality assume that $f_e(x,y)\geq 0$.
Choose $\varphi_{e} = \psi_{e}\equiv 0$. Then $\varphi_e(x) + \psi_e(y) =0 \leq f_{e}(x,y)$ for all $e \in \cE$, $(x,y) \in \Omega^2$ and $G(\varphi, \psi) < \infty$. In addition we have that $A_u(\varphi,\psi) \equiv 0$.
Note that $F(0) = -\sum_{u \in \cV} \inf_{\mu \in \cP(\Omega)} \int_\Omega f_u \dd \mu = -\sum_{u \in \cV}\min_{x \in \Omega} f_u(x) < \infty$. 
%Next we are going to show that $F$ is continuous in the sup-norm topology at $0$. 
%Since $F$ is a separable finite sum it suffices to show continuity of $\sigma_{\cP(\Omega)}(\cdot - f_u)$ at $0$ for any $u \in \cV$.
%Let $u \in \cV$. 
%Consider a sequence $\varphi_u^k \in \cC(\Omega)$ with $\varphi_u^k \to 0$ in the sup-norm topology. This means that $\|\varphi_u^k - 0\|_{\infty}=\max_{x \in \Omega} |\varphi_u^k(x)| \to 0$ ($\varphi_u^k$ converges uniformly to $0$). Let $\varepsilon >0$. Then for $k$ sufficiently large we have $|\varphi_u^k(x)| \leq \varepsilon$ for all $x \in \Omega$.
%Then we have $f_u(x) -\varepsilon \leq f_u(x) -\varphi_u^k(x) \leq f_u(x) + \varepsilon$ for all $x \in \Omega$. Denote by $f_u^*:= -\sigma_{\cP(\Omega)}(0-f_u) =\min_{x \in \Omega} f_u(x)$, let $x^*$ such that $f_u(x^*) = f_u^*$ and let $x^k$ such that $-\sigma_{\cP(\Omega)}(\varphi_u^k-f_u)=\min_{\mu \in \cP(\Omega)} \int_\Omega (f_u-\varphi_u^k) \dd \mu=\min_{x \in \Omega} (f_u - \varphi_u^k)(x)= (f_u-\varphi_u^k)(x^k)$ where existence of the minimizers follows by compactness of $\Omega$ and lower semicontinuity of $f_u$ and $f_u - \varphi_u^k$. We have $f^* -\varepsilon \leq f_u(x^k) -\varepsilon \leq f_u(x^k) -\varphi_u^k(x^k)$ and $f_u(x^k) -\varphi_u^k(x^k) \leq f_u(x^*) -\varphi_u^k(x^*) \leq f^* +\varepsilon$ implying that $|-\sigma_{\cP(\Omega)}(\varphi^k - f_u) -f^*| \leq \varepsilon$. This implies that $\sigma_{\cP(\Omega)}(\varphi^k - f_u) \to -f^*=\sigma_{\cP(\Omega)}(0 - f_u)$ for $k \to \infty$.
Clearly, $\sigma_{\cP(\Omega)}(\cdot - f_u)$ is continuous in the sup-norm topology at $0$ for any $u \in \cV$. Since $F$ is a separable finite sum it is continuous at $0$ as well.

By \cite[Theorem 1]{rockafellar1967duality} we obtain that $\inf_{(\varphi,\psi) \in (\cC(\Omega)^\cE)^2} F(A(\varphi, \psi)) + G(\varphi, \psi)$ is stably set. Thus we can invoke \cite[Theorem 3]{rockafellar1967duality} and obtain that
\begin{equation*}
\inf_{(\varphi,\psi) \in (\cC(\Omega)^\cE)^2} F(A(\varphi, \psi)) + G(\varphi, \psi)=\sup_{\mu \in \cP(\Omega)^\cV} -G^*(-A^*(\mu)) -\sum_{u \in \cV} \langle f_u, \mu_u \rangle
\end{equation*}
and there exists $\nu\in \cP(\Omega)^\cV$ such that 
\[
\sup_{\mu \in \cP(\Omega)^\cV} -G^*(-A^*(\mu)) -\sum_{u \in \cV} \langle f_u, \mu_u \rangle=-G^*(-A^*(\nu)) -\sum_{u \in \cV} \langle f_u, \nu_u \rangle. \qedhere
\] 
\end{proof}

\section{Proof of \cref{thm:duality_existence}}
\label{secsup:duality_existence_proof}

\begin{proof} \label{prf:duality_existence}
Assume that $f_{e} : \Omega \times \Omega \to \bR$ are continuous with modulus $\omega_e$ for all $e \in \cE$.
Let $(\varphi, \psi) \in \cC(\Omega)^\cE \times \cC(\Omega)^\cE$ satisfying $(\varphi_e, \psi_e) \in \cK_e$. Consider the score 
\begin{align}
-F(A(\varphi, \psi)) &= \min_{\mu \in \cP(\Omega)^\cV} \sum_{u \in \cV}  \int_\Omega (f_u -A_u(\varphi, \psi)) \dd \mu_u \\
&=\min_{\mu \in \cP(\Omega)^\cV} \sum_{u \in \cV}  \langle f_u, \mu_u \rangle + \sum_{(u,v) \in \cE} \langle \varphi_{(u,v)},\mu_u \rangle + \langle \psi_{(u,v)}, \mu_v \rangle \label{eq:score}.
\end{align}
Since $(\varphi_e, \psi_e) \in \cK_e$ we have that $\psi_e(y) \leq f_e(x,y) - \varphi_e(x)$ for all $(x,y) \in \Omega^2$ implying that 
\begin{align}\label{eq:ineq_beta}
\psi_e(y) \leq \inf_{x \in \Omega} f_{e}(x,y) - \varphi_e(x) = (\varphi_e)^{f_e}(y),
\end{align}
and we have $\varphi_e(x) + (\varphi_e)^{f_e}(y) \leq f_e(x,y)$ for all $(x,y) \in \Omega^2$.
Similarly we obtain that
\begin{align}\label{eq:ineq_alpha}
\varphi_e(x) \leq \inf_{y \in \Omega} f_{e}(x,y) - (\varphi_e)^{f_e}(y) =  (\varphi_e)^{f_e \bar f_e}(x),
\end{align}
and
\begin{equation*}
(\varphi_e)^{f_e \bar f_e}(x) + (\varphi_e)^{f_e}(y) \leq f_e(x,y).
\end{equation*}
Since by assumption $f_e$ is continuous with modulus $\omega_e$, both $(\varphi_e)^{f_e \bar f_e}, (\varphi_e)^{f_e}$ inherit the same modulus \cite[Box 1.8]{santambrogio2015optimal} and are in particular continuous.
Overall this means that $((\varphi_e)^{f_e \bar f_e}, (\varphi_e)^{f_e}) \in \cK_e$ is still feasible.
In addition the inequalities \cref{eq:ineq_alpha} and \cref{eq:ineq_beta} imply that for every $\mu \in \cP(\Omega)$ we have that $\int_\Omega \varphi_e \dd \mu \leq \int_\Omega (\varphi_e)^{f_e \bar f_e} \dd \mu$ and $\int_\Omega \psi_e \dd \mu \leq \int_\Omega (\varphi_e)^{f_e} \dd \mu$. Since $-A_u(\varphi, \psi)$ is a conic combination of multipliers $\varphi_e$ and $\psi_e$ %$\int_\Omega -A_u(\varphi,\psi) \dd \mu \leq \int_\Omega -A_u((\varphi_e)^{f_e f_e}, (\varphi_e)^{f_e}) \dd\mu$ implying that
the infimum in $-F(A(\varphi, \psi))$ increases if anything when $\varphi_e$ is replaced with $(\varphi_e)^{f_e \bar f_e}$ and $\psi_e$ is replaced with $(\varphi_e)^{f_e}$.

For each $(\varphi_e, \psi_e) \in \cK_e$ we can add and subtract a constant $c_e$ such that $(\varphi_e - c_e, \psi_e+c_e)$ achieves the same score \cref{eq:score}. Choose $c_e=\min_{x \in \Omega} \varphi_e(x)$ then $\varphi_e(x) - c_e \in [0, \omega_e(\diam \Omega)]$ and $\psi_e(x)+c_e \in [\min f_e - \omega_e(\diam \Omega), \max f_e]$, i.e., $\varphi_e - c_e, \psi_e+c_e$ are both uniformly bounded.

Now consider a maximizing sequence $(\varphi^k,\psi^k)$ of \cref{eq:local_marginal_polytope_dual}. In particular this means $(\varphi_e^k,\psi_e^k) \in \cK_e$. Then $(\varphi_e^k,\psi_e^k)$ can be replaced by $((\varphi_e^k)^{f_e \bar f_e} - c_e^k, (\varphi_e^k)^{f_e}+c_e^k)$ for $c_e^k:=\min_{x \in \Omega} (\varphi_e^k)^{f_e \bar f_e}(x)$ which is still a maximizing sequence whose elements are uniformly bounded and equicontinuous.
By the Arzelà-Ascoli Theorem \cite[Theorem 11.28]{rudin1987real} there exists a uniformly converging subsequence. Since the objective (and the constraints) are upper semicontinuous the limit must be a maximizer.
\end{proof}

\section{Proof of \cref{thm:polynomial_duality_gap_vanishes}}
\label{secsup:polynomial_duality_gap_vanishes_proof}

\begin{proof} \label{prf:polynomial_duality_gap_vanishes}
First note that $\cref{eq:dual_relax_alt}-\cref{eq:dual_relax_poly} \geq 0$.
Let $(\alpha,\beta)$ be the maximizer of \cref{eq:dual_relax_alt} which exists thanks to \cref{thm:duality_existence} and the equality $\cref{eq:local_marginal_polytope_dual}=\cref{eq:dual_relax_alt}$. Let $(\gamma,\psi) \in \cC(\Omega)^\cE \times \cC(\Omega)^\cE$. % and $(\gamma, \psi)$ be a maximizer of \cref{eq:dual_relax_poly}.
Then we have for every $x \in \Omega$:
\begin{align*}
f_{u}(x) - A_u(\alpha, \beta)(x) &= f_{u}(x) - A_u(\gamma, \psi)(x) + A_u(\gamma,  \psi)(x) -  A_u(\alpha, \beta)(x) \\
&\leq f_{u}(x) - A_u(\gamma,  \psi)(x) + \|A_u(\gamma,  \psi) - A_u(\alpha, \beta)\|_\infty.
\end{align*}
Minimizing both sides wrt $x\in \Omega$ and summing over $u \in \cV$ this implies that
\begin{align} \label{eq:ineq_unary}
\sum_{u \in \cV} \min_{x \in \Omega} (f_{u} - A_u(\alpha, \beta))(x) - \min_{x \in \Omega} (f_{u} - A_u( \gamma,  \psi))(x) \leq \sum_{u \in \cV} \|A_u(\gamma,  \psi) - A_u(\alpha, \beta)\|_\infty.
\end{align}
Furthermore we have
\begin{align*}
 \sum_{u \in \cV}\|A_u(\gamma,  \psi) - A_u(\alpha, \beta)\|_\infty &= \sum_{u \in \cV}\max_{x \in \Omega} \left| \sum_{v :(u,v) \in \cE} (\alpha_{(u,v)} - \gamma_{(u,v)})(x) +(\beta_{(v,u)}- \psi_{(v,u)})(x) \right| \\
 &\leq \sum_{u \in \cV}\max_{x \in \Omega} \sum_{v :(u,v) \in \cE} |(\alpha_{(u,v)} - \gamma_{(u,v)})(x)| \\
 &\qquad + \sum_{v :(u,v) \in \cE} |(\beta_{(v,u)}- \psi_{(v,u)})(x)| \\
 &\leq \sum_{e \in \cE} \|\alpha_{e} - \gamma_{e}\|_{\infty} +  \|\beta_{e}- \psi_{e}\|_\infty.
% &\leq \deg_{\mathrm{out}}(u) \max_{e \in \cE} \|\alpha_{e} - \gamma_{e}\|_{\infty} + \deg_{\mathrm{in}}(u) \|\beta_{e}- \psi_{e}\|_\infty
\end{align*}
We also have for every $x,y \in \Omega$
\begin{align*}
 f_{e}(x,y) - \alpha_{e}(x) - \beta_{e}(y) &= f_{e}(x,y) -  \gamma_{e}(x) -  \psi_{e}(y) +  \gamma_{e}(x) - \alpha_{e}(x) +  \psi_{e}(y) - \beta_e(y) \\
&\leq f_{e}(x,y) -  \gamma_{e}(x) -  \psi_{e}(y) + \| \gamma_{e}- \alpha_{e}\|_\infty + \|  \psi_{e} - \beta_e \|_\infty.
\end{align*}
Minimizing both sides wrt $x,y\in \Omega$ and summing over $e \in \cE$ this implies that
\begin{align}
\sum_{e \in \cE} \min_{x,y \in \Omega}  &f_{e}(x,y) - \alpha_{e}(x) - \beta_{e}(y) - \min_{x,y \in \Omega} f_{e}(x,y) -  \gamma_{e}(x) -  \psi_{e}(y) \notag \\
& \qquad \leq \sum_{e \in \cE} \| \gamma_{e}- \alpha_{e}\|_\infty + \|  \psi_{e} - \beta_e \|_\infty. \label{eq:ineq_pairw}
\end{align}
Now choose $\varepsilon >0$. Since $\Omega$ is compact and $\bR[X]$ is a subalgebra of $\cC(\Omega)$ by the Stone--Weierstrass theorem we can find polynomials $\gamma_e,\psi_e \in \bR[X]$ such that $\| \gamma_{e}- \alpha_{e}\|_\infty \leq \frac\varepsilon{4|\cE|}$ and $\| \psi_{e}- \beta_{e}\|_\infty \leq \frac\varepsilon{4|\cE|}$. Let $n = \max_{e \in \cE} \max\{\deg(\gamma_{e}), \deg(\psi_{e})\}$. 
Summing \cref{eq:ineq_unary} and \cref{eq:ineq_pairw} since $(\alpha, \beta)$ is a maximizer of \cref{eq:dual_relax_alt} we have
\begin{align*}
\cref{eq:dual_relax_alt}-\cref{eq:dual_relax_poly} &\leq 2\sum_{e \in \cE} \| \gamma_{e}- \alpha_{e}\|_\infty + \|  \psi_{e} - \beta_e \|_\infty \\ &\leq 2\sum_{e \in \cE}\frac\varepsilon{2|\cE|} = \varepsilon.
\end{align*}
Thanks to \cref{thm:duality_general} we have $\cref{eq:dual_relax_alt}=\cref{eq:local_marginal_polytope_dual}=\cref{eq:local_marginal_polytope}$. Since $\cref{eq:dual_relax_poly}=\cref{eq:local_marginal_polytope_dual_discrete}$ this means that for any $\varepsilon >0$ we can find $n \in \bN$ sufficiently large so that
\[
0 \leq\cref{eq:local_marginal_polytope}- \cref{eq:local_marginal_polytope_dual_discrete} \leq \varepsilon.
\]
If, in addition, $(\cV,\cE)$ is a tree we have in light of \cref{thm:thightness_tree} that $\cref{eq:local_marginal_polytope} = \cref{eq:mrf}$.
\end{proof}

\section{Proof of \cref{thm:GeneralmetricProblemConvergenceRate}}
\label{secsup:GeneralmetricProblemConvergenceRate_Proof}

\begin{proof} \label{prf:GeneralmetricProblemConvergenceRate}
In view of the proof of the previous theorem we consider the equivalent unconstrained problem $\cref{eq:dual_relax_alt}=\cref{eq:local_marginal_polytope_dual}$. Let $(\alpha,\beta)$ be the maximizer of \cref{eq:dual_relax_alt}. As shown in the proof of the previous theorem the duality gap can be bounded as
\begin{equation*}
\cref{eq:dual_relax_alt}-\cref{eq:dual_relax_poly} \leq 2\sum_{e \in \cE} \| \gamma_{e}- \alpha_{e}\|_\infty + \|  \psi_{e} - \beta_e \|_\infty.
\end{equation*}
Thanks to \cref{thm:duality_existence} $\alpha_e, \beta_e \in \cC(\Omega)$ inherit the modulus of continuity $\omega_e$.
% and we can rewrite
%$$
%\| \gamma_{e}- \alpha_{e}\|_\infty = \sup_{\theta \in [0, 2\pi]} \| \gamma_{e}(\cos(\theta), \sin (\theta))- \alpha_{e}(\cos(\theta), \sin (\theta))\|_\infty,
%$$
%and
%$$
%\| \psi_{e}- \alpha_{e}\|_\infty = \sup_{\theta \in [0, 2\pi]} \| \psi_{e}(\cos(\theta), \sin (\theta))- \alpha_{e}(\cos(\theta), \sin (\theta))\|_\infty.
%$$
Now fix $n\in \bN$. Thanks to \cite[Theorem 2]{NewmanShapiroJacksonsTheorem} there exist polynomials $\rho_e$ and $\lambda_e$ with maximum degree $n$ such that $\| \rho_{e}- \alpha_{e} \|_\infty \leq C_{m} \omega_e(1/n)$ and $\|  \lambda_{e} - \beta_e \|_\infty \leq C_{m} \omega_e(1/n)$. Thus
\[
\cref{eq:dual_relax_alt}-\cref{eq:dual_relax_poly} \leq 4 C_{m} \sum_{e \in \cE} \omega_e(1/n).
\]
If, in addition, $(\cV,\cE)$ is a tree we have in light of \cref{thm:thightness_tree} that $\cref{eq:local_marginal_polytope} = \cref{eq:mrf}$.
\end{proof}

\section{Convergence of the optimal transport subproblem}

\begin{figure}[t!]
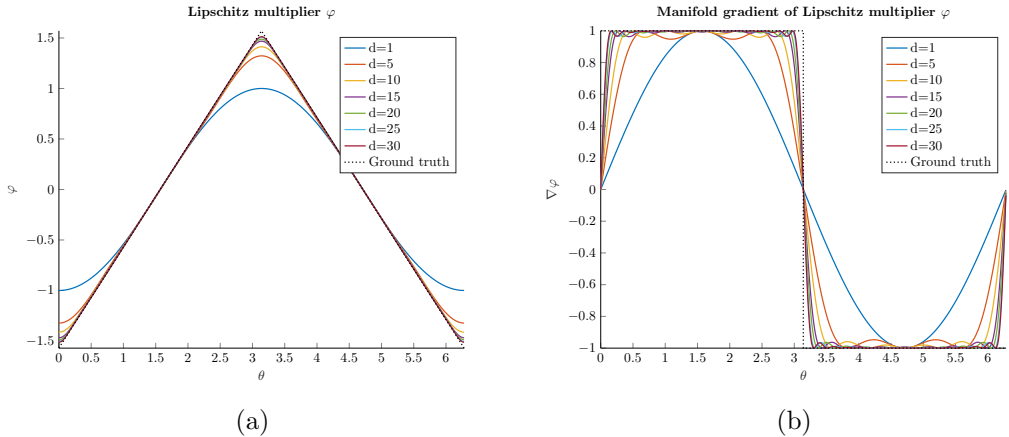

\centering
\begin{subfigure}[b]{0.45\textwidth}
\centering
\resizebox {\textwidth} {!} {
\input{figures/tikz/ALagrangeMultiplier2Vertex2}
}
\caption{}
\end{subfigure}
\begin{subfigure}[b]{0.45\textwidth}
\centering
\resizebox{\textwidth} {!} {
\input{figures/tikz/LagrangeMultiplier2Vertex2}
}
\caption{}
\end{subfigure}
\caption{Lipschitz continuous Lagrange multiplier $\varphi$ (optimal) and $\psi$ (reference), as well as their manifold gradients, for a connection between two vertices on opposing points on the unit circle for various degrees.}
\label{fig:exp_OT_approximation}
\end{figure}

\begin{figure}[t!]
\centering
\begin{subfigure}[b]{0.45\textwidth}
\resizebox {\textwidth} {!} {
\definecolor{mycolor1}{rgb}{0.00000,0.00000,0.00000}%
\definecolor{mycolor2}{rgb}{0.00000,0.00000,0.00000}%
\begin{tikzpicture}

\begin{axis}[%
width=4.521in,
height=3.566in,
at={(0.758in,0.481in)},
scale only axis,
xmin=0,
xmax=30,
xlabel style={font=\color{white!15!black}},
xlabel={n},
ymin=0,
ymax=1.5707963267949,
ylabel style={font=\color{white!15!black}},
ylabel={approximation error},
axis background/.style={fill=white},
title style={font=\bfseries},
title={approximation error},
axis x line*=bottom,
axis y line*=left,
legend style={legend cell align=left, align=left, draw=white!15!black}
]
\addplot [dotted, color=mycolor1, line width=1.0pt]
  table[row sep=crcr]{%
2	1.34096002127558\\
6	0.628612242171849\\
10	0.412266061373133\\
16	0.272478492448702\\
20	0.222363813541045\\
26	0.174340100796799\\
30	0.152416230654856\\
};
\addlegendentry{$\psi$}

\addplot [color=mycolor2, line width=1.0pt]
  table[row sep=crcr]{%
2	1.14159266113216\\
6	0.495882850394847\\
10	0.316182566881431\\
16	0.204750595917997\\
20	0.165783225942775\\
26	0.128966446192493\\
30	0.116409565011256\\
};
\addlegendentry{$\varphi$}

\end{axis}

\begin{axis}[%
width=5.833in,
height=4.375in,
at={(0in,0in)},
scale only axis,
xmin=0,
xmax=1,
ymin=0,
ymax=1,
axis line style={draw=none},
ticks=none,
axis x line*=bottom,
axis y line*=left
]
\end{axis}
\end{tikzpicture}%
}
\caption{}
\end{subfigure}
\begin{subfigure}[b]{0.45\textwidth}
\resizebox {\textwidth} {!} {

\definecolor{mycolor1}{rgb}{1.00000,0.00000,0.50000}%
\definecolor{mycolor2}{rgb}{0.88889,0.05556,0.50000}%
\definecolor{mycolor3}{rgb}{0.77778,0.11111,0.50000}%
\definecolor{mycolor4}{rgb}{0.55556,0.22222,0.50000}%
\begin{tikzpicture}

\begin{axis}[%
width=4.521in,
height=3.344in,
at={(0.758in,0.703in)},
scale only axis,
xmin=1,
xmax=8,
xtick={1,2,3,4,5,6,7,8},
xticklabels={{0.392699},{0.785398},{1.1781},{1.5708},{1.9635},{2.35619},{2.74889},{3.14159}},
xticklabel style={rotate=30},
xlabel style={font=\color{white!15!black}},
xlabel={angle ($\theta$)},
ymode=log,
ymin=1e-10,
ymax=1,
yminorticks=true,
ylabel style={font=\color{white!15!black}},
ylabel={relative gap $G_{R}$},
axis background/.style={fill=white},
title style={font=\bfseries},
title={relative duality gap},
axis x line*=bottom,
axis y line*=left,
legend style={legend cell align=left, align=left, draw=white!15!black}
]
\addplot [color=mycolor1, line width=1.0pt]
  table[row sep=crcr]{%
1	0.00641315011770382\\
2	0.0255046502926486\\
3	0.0568346870610717\\
4	0.0996836859613748\\
5	0.153072010418994\\
6	0.215786697595359\\
7	0.286414512288238\\
8	0.363380224094563\\
};
\addlegendentry{d=1}

\addplot [color=mycolor2, line width=1.0pt]
  table[row sep=crcr]{%
1	0.00641314875072842\\
2	0.0255046193084505\\
3	0.056834653484721\\
4	0.0996836803299816\\
5	0.153072006266255\\
6	0.215786685879913\\
7	0.286414510144412\\
8	0.363380218779109\\
};
\addlegendentry{d=2}

\addplot [color=mycolor3, line width=1.0pt]
  table[row sep=crcr]{%
1	4.98085607279691e-05\\
2	0.000811559766867151\\
3	0.00421854983920215\\
4	0.0137528892654821\\
5	0.0345555564499472\\
6	0.0728824769903062\\
7	0.134238777469713\\
8	0.220303209604985\\
};
\addlegendentry{d=3}

\addplot [color=teal!25!mycolor2, line width=1.0pt]
  table[row sep=crcr]{%
1	3.21441167588406e-05\\
2	0.000568078951891698\\
3	0.00328229496891978\\
4	0.0118486066930781\\
5	0.0321852165726344\\
6	0.0711349307542617\\
7	0.133679161623393\\
8	0.220303193814122\\
};
\addlegendentry{d=4}

\addplot [color=mycolor4, line width=1.0pt]
  table[row sep=crcr]{%
1	3.9470905599679e-07\\
2	2.68610531019422e-05\\
3	0.000330371086094451\\
4	0.00203381034183698\\
5	0.00851876899757317\\
6	0.0277107679837781\\
7	0.0731912416505712\\
8	0.157844413250503\\
};
\addlegendentry{d=5}

\addplot [color=teal!50!mycolor2, line width=1.0pt]
  table[row sep=crcr]{%
1	4.38721613379488e-09\\
2	7.9183735850464e-07\\
3	3.88006148071518e-05\\
4	0.000845685123691041\\
5	0.00594288068728711\\
6	0.0247822967554179\\
7	0.0719447896174981\\
8	0.157844412933204\\
};
\addlegendentry{d=6}

\addplot [color=teal!40!mycolor4, line width=1.0pt]
  table[row sep=crcr]{%
1	9.54251219230561e-09\\
2	7.10176202596449e-07\\
3	2.32352028152954e-05\\
4	0.000289939723559033\\
5	0.00214434335397098\\
6	0.0110143955583473\\
7	0.0427585242640342\\
8	0.122927516589589\\
};
\addlegendentry{d=7}

\addplot [color=teal!75!mycolor2, line width=1.0pt]
  table[row sep=crcr]{%
1	1.22630019008232e-08\\
2	3.56541608921297e-08\\
3	1.67443331487365e-06\\
4	3.66141922485897e-05\\
5	0.000555673096035944\\
6	0.00778394772557445\\
7	0.0408624554014367\\
8	0.122927520802726\\
};
\addlegendentry{d=8}

\addplot [color=teal!80!mycolor4, line width=1.0pt]
  table[row sep=crcr]{%
1	1.38283805872866e-08\\
2	6.3630787623858e-08\\
3	9.06201912447998e-07\\
4	3.24733533722598e-05\\
5	0.00050202088913442\\
6	0.00444617395963919\\
7	0.0259119043808876\\
8	0.100644043063443\\
};
\addlegendentry{d=9}

\addplot [color=teal, line width=1.0pt]
  table[row sep=crcr]{%
1	2.83958904183565e-08\\
2	1.5363864604188e-08\\
3	1.65841412901849e-07\\
4	5.58697922962331e-06\\
5	0.000113124519625105\\
6	0.00196089894369785\\
7	0.0234856781634952\\
8	0.100644086268075\\
};
\addlegendentry{d=10}

\end{axis}

\begin{axis}[%
width=5.833in,
height=4.375in,
at={(0in,0in)},
scale only axis,
xmin=0,
xmax=1,
ymin=0,
ymax=1,
axis line style={draw=none},
ticks=none,
axis x line*=bottom,
axis y line*=left
]
\end{axis}
\end{tikzpicture}%
}
\caption{}
\end{subfigure}
\caption{(a) Error between polynomial approximation of the geodesic distance by $\varphi$ (optimal) and $\psi$ (reference) and actual geodesic distance ($\pi$) for diametrically opposite points in terms of degree $n$, (b) Relative duality gap for optimal multiplier $\varphi$ for the optimal transport problem in function of degree $n$ and angle $\theta$ between $x$ and $y$.}
\label{fig:exp_OT_relG}
\end{figure}
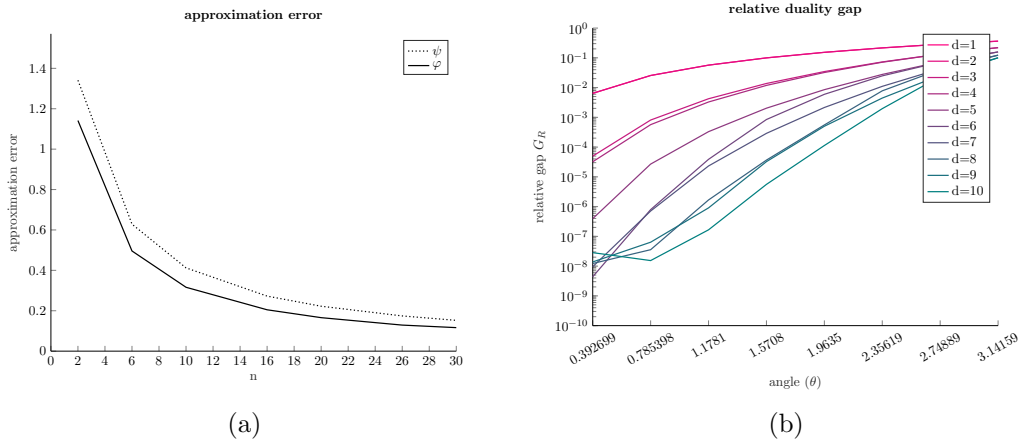

In this additional numerical experiment, we evaluate the modeling power of polynomials for solving the optimal transport problem, which is a subproblem of our formulation. Therefore we solve the problem
\begin{equation}
\begin{aligned}
\max_{\varphi \in \PEN{x}{n}} \quad & \langle \varphi, \mu_{2} - \mu_{1} \rangle \\
\text{subject to} \quad & L - \langle \xi, J^{*}(x) \nabla \varphi(x) \rangle \geq 0, \quad \forall (x, \xi) \in \Omega \times \mS{p}
\end{aligned}
\end{equation}
for $\mu_{1}, \mu_{2} \in \cM$ fixed. For this experiment we choose two diametrically opposite points $x, y$ and let $\Omega = \mS{1}$, $\mu_{1} = \delta_{x}$ and $\mu_{2} = \delta_{y}$, as this reflects a certain worst-case scenario in our relaxation. We compute and plot the polynomial multiplier $\varphi$ restricted to $\mS{1}$ in terms of the angle $\theta$, as well as its manifold gradient, in \cref{fig:exp_OT_approximation}.\\
\\
Furthermore, we display the gap between the polynomial approximations $\varphi$ in terms of the degree $n$, and the actual geodesic distance between $x$ and $y$ ($d_{g}(x, y)=\pi$). We also plot the gap for the polynomial approximations $\psi$ used in the proof of theorem \cref{thm:geodesic_approximation}. These results are found in \cref{fig:exp_OT_relG}, where for degrees $n$ ranging from $1$ to $10$ the relative error between the geodesic distance and our polynomial approximation for various angles $\theta$ between $x$ and $y$ is also shown. Whenever $x$ and $y$ are not diametrically opposite, the approximation is always notably better than what is shown in the experiment; hence the particular focus on diametrically opposite points which pose a worst case problem for $\mS{m}$ in general.

\end{document}